\documentclass{article}
\usepackage{color} 
\usepackage{multirow} 
\usepackage{amssymb,bm,amsmath,dsfont,amsthm}
\usepackage{latexsym}
\usepackage{epsfig}
\usepackage{url}
\usepackage{mdwlist}

\newtheorem{theorem}{Theorem}
\newtheorem{Lemma}[theorem]{Lemma}
\newtheorem{Proposition}[theorem]{Proposition}

\newtheorem{Condition}[theorem]{Condition}

\newcommand{\Pb}{{\mathds{P}}}
\newcommand{\Eb}{{\mathds{E}}}
\newcommand{\Indb}{{\mathds{1}}}

\newcommand{\NN}{{\mathbb{N}}}

\newcommand{\eps}{\varepsilon}

\newcommand{\cN}{{\mathcal N}}
\newcommand{\tcP}{\tilde{\mathcal P}}

\newcommand{\cC}{{\mathcal C}}

\newcommand{\cK}{{\mathcal K}}
\newcommand{\tC}{{\tilde{\mathcal C}}}
\newcommand{\tS}{{\tilde{S}}}
\newcommand{\tV}{{\tilde{V}}}
\newcommand{\tH}{{\tilde{H}}}
\newcommand{\tF}{{\tilde{F}}}

\newcommand{\bcC}{\breve{\mathcal{C}}}
\newcommand{\bC}{{\bar{C}}}

\def\tp{\tilde{p}}

\def\tn{\tilde{n}}

\def\tl{\tilde{\lambda}}
\def\tgamma{\tilde{\gamma}}

\def\tG{\tilde{G}}
\def\tD{\tilde{D}}

\def\bpi{\boldsymbol{\pi}}

\def\balpha{\boldsymbol{\alpha}}

\def\boldd{\boldsymbol{d}}
\def\boldk{\boldsymbol{k}}
\def\boldp{\boldsymbol{p}}

\def\boldg{\boldsymbol{\gamma}}
\def\boldt{\boldsymbol{t}}

\newcommand{\tGr}{\tilde{G}\left(n,\boldd,\boldg\right)}
\newcommand{\tGrg}{\tilde{G}\left(n,\boldd,\gamma\right)}
\newcommand{\tGrr}{\tilde{G}^*\left(n,\boldd,\boldg\right)}

\newcommand{\Gr}{G\left(n,\boldd\right)}
\newcommand{\Grr}{G^*\left(n,\boldd\right)}

\newcommand{\ie}{{\it i.e.~}}

\def\N{\mathbb{N}}
\newcommand{\cd}{(\gamma_r)_{r=0}^{\infty}}
\newcommand{\di}{(d_i)_1^n}

\newcommand{\intn}{\{1,...,n\}}

\newcommand{\pr}{(p_{r})_{r=0}^{\infty}}

\newcommand{\proj}{\phi}
\newcommand{\projj}{\phi}

\title{How Clustering Affects Epidemics \\ in Random Networks}
\author{Emilie Coupechoux, {\textit{INRIA - ENS}}\\
  Marc Lelarge, {\textit{INRIA - ENS}}\\
E-mail: Emilie.Coupechoux, Marc.Lelarge @ens.fr}

\begin{document}

\maketitle


\begin{abstract}
Motivated by the analysis of social networks, we study a model of random networks that has both a given degree distribution and a tunable clustering coefficient. We consider two types of growth processes on these graphs: diffusion and symmetric threshold model. 
The diffusion process is inspired from epidemic models. It is characterized by an infection probability, each neighbor transmitting the epidemic independently.
In the symmetric threshold process, the interactions are still local but the propagation rule is governed by a threshold (that might vary among the different nodes). An interesting example of symmetric threshold process is the contagion process, which is inspired by a simple coordination game played on the network.
Both types of processes have been used to model spread of new ideas, technologies, viruses or worms and results have been obtained for random graphs with no clustering. In this paper, we are able to analyze the impact of clustering on the growth processes. While clustering inhibits the diffusion process, its impact for the contagion process is more subtle and depends on the connectivity of the graph: in a low connectivity regime, clustering also inhibits the contagion, while in a high connectivity regime, clustering favors the appearance of global cascades but reduces their size. 

For both diffusion and symmetric threshold models, we characterize conditions under which global cascades are possible and compute their size explicitly, as a function of the degree distribution and the clustering coefficient. Our results are applied to regular or power-law graphs with exponential cutoff and shed new light on the impact of clustering.
\end{abstract}


\section{Introduction}
\label{sec:intro}

Many network phenomena are well modeled as spreads of
epidemics through a network. However, depending on the motivation, different mechanisms are at work and we believe that different models should be used. 
For example, to model the spread of worms or email viruses, and, more generally, faults, a simple diffusion model is often used, where each node when it becomes infected or faulty 'transmits' the infection to her neighbors independently and with a given probability. 
There is now a vast literature on such epidemics on complex networks (see
\cite{new-book03} for a review).
But if nodes represent agents in a social network, the transmission mechanism is independent of the local condition
faced by the agents concerned.
But to model the spread of innovations as a social process, individual's adoption behavior is highly correlated with the behavior of her contacts. In this case, there is a factor of persuasion or coordination involved and relative considerations tend to be important
in understanding whether some new behavior or belief is adopted \cite{vega07}. 
In social contexts, the spread of information and behavior often
exhibits features that do not match well those of the diffusion model just described where an individual is influenced by each of her neighbors independently. In such a context, the linear threshold model, originally proposed by Granovetter \cite{G78}, captures in a simple way the local correlations among individuals. We will study the symmetric threshold model proposed by Lelarge \cite{lel:diff}, which generalizes both bootstrap percolation \cite{BP:bootstrap} and contagion model \cite{mor}.

In this paper, we will analyze two different types of epidemics modeled by simple growth processes that we now describe. In a given network, each node can either be active or inactive.
The {\em diffusion} process corresponds to the case where each node of the network that becomes active transmits the activation to her neighbors with a given probability, independently from each others. 
On the other hand, in the {\em symmetric threshold} model, a threshold is associated to each node, and the dynamics of the process corresponds to the case where a node of the network becomes active as soon as the number of her active neighbors exceeds the threshold of the node. As in the original model of \cite{lel:diff}, thresholds are (possibly) random, with a distribution depending on the degree of the node, and such that thresholds are independent among nodes. The symmetric threshold model will allow us to analyze the contagion process \cite{mor}.
For both models, we will first consider a case where there is only one initial active node and characterize conditions under which global cascades are possible, that is to say when a positive fraction of the population is active at the end of the process. In such cases we compute the probability of a cascade and its size.
Then, we consider the cascade size when a positive fraction of the population is initially active. The initial activations are random in that case, and the probability that a node belongs to the seed might depend on its degree.

We now describe the model of random graphs studied in this paper. For many real-world networks, the underlying graph $G$ is a power-law graph, \ie a
graph whose degree distribution follows a power law. Random graphs with a given degree sequence allow to model such behavior. This model is usually called the configuration model \cite{bb}. There is a vast literature on the analysis of the diffusion for such graphs \cite{new-book03}. The contagion process has also been studied for such graphs through heuristics \cite{wat02} or rigorously in \cite{lel:diff} or \cite{amini:10}.
Random graphs are not considered to be highly realistic models of most real-world networks, and they are used as first approximation as they are a natural choice for sparse interaction network in the absence of any known geometry.
One essential drawback of this model is that these graphs are 'locally tree-like': short cycles are very rare.
However, real-world networks are often highly clustered, meaning that there
is a large number of triangles and other short cycles \cite{new03}.
For social graphs, this is a consequence of the fact that friendship circles are typically strongly overlapping so that many of our friends are also friends of each other.

There are works in the physics and biology literature on models of random graphs with clustering \cite{new09}. Our model is inspired from \cite{trapman07} which allow to model random graphs with positive clustering and possibly power law degree distribution. The idea is to 'add' clustering to a standard configuration model by replacing some vertices by cliques.
By choosing the fraction of vertices replaced, this leads to a graph
where the amount of clustering can be tuned by adjusting the
parameters of the model. This model generalizes the standard configuration
model to incorporate clustering. Understanding how clustering affects diffusion and contagion remains largely an open question.
Our work is a first step towards addressing this issue in a systematic and rigorous way. In particular, we are able to make a rigorous analysis of the impact of a variation of the clustering coefficient while keeping the degree distribution in the graph fixed. To the best of our knowledge, this sensitivity analysis is new and gives a number of insights on the impact of clustering.
In \cite{trapman07} and \cite{GleesonMH:clustering}, the diffusion process on such graphs is analyzed by an heuristic approximation through a branching process with additional cliques. We derive rigorous proofs for these results.
A different model of random graphs with clustering, called random intersection graphs has been studied rigorously in \cite{DeijfenKets} and in \cite{britton-2007}, the diffusion process is studied on such graphs. However, the degree distribution for this kind of graphs has to be a Poisson distribution and clustering cannot vary independently of the degree distribution.
Up to our knowledge, results on the contagion model have not been proved before our work for random graphs with clustering.
\footnote{ A preliminary version without proofs of our work appeared in \cite{CL:Netgcoop}.} 
Recently, \cite{AOY} derives bounds which are valid for the contagion model \cite{mor} on deterministic networks. Our analysis in contrast gives asymptotic results as the size of the graph tends to infinity and allows us (by looking at a more specific model) to identify neatly the impact of clustering.


The paper is organized as follows. In Section \ref{sec:rg}, we present the graph model, compute its asymptotic degree distribution and its asymptotic clustering coefficient. We explain how to tune the clustering coefficient of the model, while keeping the asymptotic degree distribution fixed. In Section \ref{sec:diff}, we derive the minimal value for the probability of infection in our random graph model with
clustering such that a global diffusion is possible, and we compute the size of the diffusion in that case. We apply these results to random regular graphs and power-law graphs and show that clustering inhibits the diffusion process in that case. We also compute the cascade size for the diffusion process with degree based activation. In Section \ref{sec:cont}, we derive the cascade condition for the symmetric threshold model on our random graph model with clustering, together with the size of the cascade when it occurs. Numerical evaluations in the particular case of the contagion process show that the effect of clustering on the contagion threshold depends on the mean degree of the graph, and that clustering decreases the cascade size when it occurs. In addition, we finally compute the cascade size for the symmetric threshold model with a slight variant of the degree based activation. Proofs are given in Section \ref{sec:pf}.

\paragraph{Notations.}

In the following, we consider asymptotics as $n\to\infty$, and we
denote by $\to_p$ the convergence in probability as $n\to\infty$. The
abbreviation 'whp' (``with high probability'') means with probability
tending to 1 as $n\to\infty$, and we use the notation $o_p(n)$, $\Theta_p(n)$ in a
standard way (see \cite{JLR:2000} for instance): $X=o_p(n)$ means that, for every $\eps>0$, $\Pb(X>\eps n)\to 0$ as $n\to\infty$.
In addition, for integers $s\geq 0$ and $0\leq r\leq s$, let $b_{sr}$ denote the binomial probabilities $b_{sr}(p):=\Pb(\textrm{Bi}(s,p)=r)={s\choose r} p^r(1-p)^{s-r}$. 

\section{Random graph model and its basic properties}\label{sec:rg}

We first present the model for the random graph, and compute its asymptotic degree distribution and its asymptotic clustering coefficient (for two different definitions).

\subsection{Model of random graph with clustering}

We first consider the uniform random graph with fixed degree distribution: since this graph has asymptotically no clustering, we will then modify it to obtain a graph with clustering.

Let $n\in\N$ and let $\boldd=(d^{(n)}_i)_{i=1}^n=\di$ be a sequence
of non-negative integers such that $\sum_i d_i$ is even. 
The integer $n$ is the number of vertices in the graph and vertex $i\in [n]$ has degree $d_i$ in the graph.
Let $\Gr$ be
a graph chosen uniformly at random among all simple (i.e. with no multi-edges or self-loops) graphs with $n$ vertices
and degree sequence $\boldd$ (assuming such graphs exist) \cite{bb}. 

We will let $n\to \infty$ and assume that we are given $\boldd$
satisfying the following regularity conditions which are standard in the random graph literature, see \cite{mr95}:
\begin{Condition}
 \label{cond}
For each $n$, $\boldd=\di$ is a sequence of non-negative integers such that $\sum_i d_i$ is even. We assume that there exists a probability distribution $\boldp=\pr$ (independent of $n$) such that:
\begin{itemize}
 \item[(i)] $n_{r}/n=\left|\{i:d_i=r\}\right|/n\to p_{r}$ as $n\to\infty$, for all $r \geq 0$;
\item[(ii)] $\lambda:=\sum_{r} rp_{r}\in(0,\infty)$;
 \item[(iii)] $\sum_i d_i^3=O(n)$.
\end{itemize}
\end{Condition}

If $D_n$ is the degree of a vertex chosen uniformly at
random among the $n$ vertices of $\Gr$, and $D$ a random variable with
distribution $\boldp$, \textit{(i)} is equivalent to the fact that $D_n
\overset{d}{\longrightarrow} D$ (convergence in distribution). In addition, \textit{(iii)} is equivalent to $\Eb[D_n^3]=O(1)$, which implies that the random variables $D_n$ are uniformly integrable or equivalently the uniform summability of $\sum_d dn_d/n$, in particular $\Eb[D_n]\to\Eb[D]$.

The model of random graphs $\Gr$ are 'locally
tree-like', i.e. they contain very few (i.e. $o(n)$) short cycles in their structure. 
We now show that it is possible to generalize this model of random graphs to
incorporate clustering in a simple way. The resulting model of random graphs will still be tractable for the analysis of diffusion and symmetric threshold models.
To 'add' clustering in $\Gr$, we replace some
vertices by a clique of size the degree in the original graph, i.e. a
vertex of degree $r$ in the original graph $\Gr$ is replaced by $r$
vertices with all the $r(r-1)/2$ edges between them and each of them
is connected to exactly one of the neighbors of the vertex in the
original graph $\Gr$ as illustrated on Figure \ref{fig:clique}.
Note that if $r=0$, i.e. if the original node is isolated, this procedure remove the node. By convention a clique of size zero is empty.

\begin{figure}
\centering
\epsfig{file=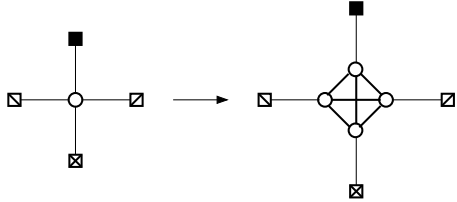, width=6cm}
\caption{Transformation}
\label{fig:clique}
\end{figure}

In order to be able to tune the clustering coefficient in the graph, we will not
replace all vertices by a clique but do a probabilistic choice whether
to replace a vertex or not:
for all $r\geq 0$, $\gamma_r\in [0,1]$ represents the probability that a vertex of degree
$r$ in $\Gr$ is replaced by a clique of size $r$ in the new
model denoted $\tGr$, where $\boldg=\cd$ is a short notation for the sequence of $\gamma_r$'s.
The choices to replace or not a vertex by a clique are made independently at each vertex of the original graph $\Gr$.
More formally, for each vertex $i\in\intn$, let $X(i)$ be a
Bernoulli random variable with parameter $\gamma_{d_i}$ (all Bernoulli random variables
being independent of each other). We construct the random graph $\tGr$ by replacing each vertex $i$ of $\Gr$ with $X(i)=1$, by a
clique of size $d_i$ where each vertex of the clique has exactly one neighbor
outside the clique being a neighbor of $i$ in the original graph. All vertices $i$ with $X(i)=0$ are unchanged.
In particular, if $\gamma_r=0$ for all $r\geq 0$, then we simply get
$\tGr=\Gr$, whereas if $\gamma_r=1$ for all $r\geq 0$, all vertices in
$\Gr$ are replaced by cliques (and isolated vertices are removed). With a little abuse of notation, we write $\tGrg$ for the graph $\tGr$ in which the sequence $\boldg$ is constant and equals to $\gamma$.

\subsection{Degree distribution in $\tGr$}
\label{subs:degree}

As we will see in the next subsection, the procedure described above
introduces clustering at soon as $\gamma_r>0$ for some $r$. It also modifies
the degree distribution in the graph and we derive the new degree
distribution here.
Recall that each vertex of
degree $r$ can either be replaced (with probability $\gamma_r$) by $r$
vertices of degree $r$, or stays as a single vertex of degree $r$
(with probability $1-\gamma_r$). The following proposition gives the
resulting asymptotic degree distribution in $\tGr$.

\begin{Proposition}\label{prop:deg}
 We consider the model $\tGr$ for a sequence $\boldd$ satisfying
 Condition \ref{cond} with probability distribution
 $\boldp=\pr$, and clustering parameter $\boldg=\cd$.
For all $r\geq 0$, let $\tn_r$ be the number of vertices with degree $r$ in $\tGr$, and let $\tn=\sum_r \tn_r$ be the total number of vertices in $\tGr$. 
Then we have, as $n\to\infty$: $$\frac{\tn}{n}\overset{p}{\longrightarrow} \tgamma:=\sum_{d\geq 0} \left[d \gamma_d+(1-\gamma_d)\right] p_d>0.$$ and, for all $r\geq 0$, the proportion of vertices with degree $r$ in $\tGr$ has the following limit, as $n\to\infty$:
\begin{eqnarray*}
 \frac{\tn_r}{\tn} \overset{p}{\longrightarrow} \tilde{p}_r:=\frac{\left[r\gamma_r+(1-\gamma_r)\right] p_r}{\tgamma}.
\end{eqnarray*}
\end{Proposition}

\begin{proof}
We first show the asymptotics for the number $\tn$ of vertices in $\tGr$.

Let $d\geq 0$, and let $B_d$ be the number of vertices with degree $d$ that are replaced by a clique. Then $B_d$ follows a Binomial distribution with parameters $(n_d,\gamma_d)$. By Condition \ref{cond}-\textit{(i)}, we have: $n_d/n\to p_d$, so that the Law of Large Numbers implies: $B_d/n \to_p \gamma_d p_d$ (which is still true if $p_d=0$).

The number of vertices with degree $d$ that are \textit{not} replaced by a clique is $n_d-B_d$, so we can express the total number $\tn$ of vertices in $\tGr$ the following way:
 \begin{eqnarray*}
  \frac{\tn}{n} = \frac{1}{n}\sum_d d B_d+(n_d-B_d)
 \overset{p}{\longrightarrow} \sum_d \left[d\gamma_d+(1-\gamma_d)\right] p_d =\tgamma
 \end{eqnarray*}
which follows from the previous limits, and the uniform summability of $\sum dn_d/n$ implied by Condition \ref{cond}-\textit{(iii)}.

Similarly, the total number of vertices with degree $r$ ($r\geq 0$) in $\tGr$ is $rB_r+(n_r-B_r)$, so the proportion of vertices with degree $r$ in $\tGr$ is:
\begin{eqnarray*}
 \frac{rB_r+(n_r-B_r)}{\tn} \overset{p}{\longrightarrow} \frac{\left[r\gamma_r+(1-\gamma_r)\right] p_r }{\tgamma} =\tp_r,
\end{eqnarray*}
which concludes the proof.
\end{proof}

In other words, if $\tD_n$ is the degree of a vertex chosen uniformly at random in $\tGr$, then Proposition \ref{prop:deg} implies that $\tD_n \overset{d}{\longrightarrow} \tD$, where $\tD$ is a random variable with distribution $(\tilde{p}_r)_{r\geq 0}$.
In the particular case where $\gamma_r=\gamma$ for all $r$, we have $\tgamma
=\gamma\lambda+1-\gamma$ and the mean degree of the graph is then
\begin{eqnarray*}
\tilde{\lambda} =\Eb[\tD] = \frac{\gamma\Eb[D(D-1)]+\Eb[D]}{\gamma\Eb[D-1]+1},
\end{eqnarray*}
which is a non-decreasing function of $\gamma$.

\subsection{Clustering coefficient}
\label{subs:clustering}

We now compute the clustering coefficient of the graph $\tGr$.
The most common definition of the clustering coefficient of a finite graph is given by:
\begin{eqnarray}\label{def:clustering}
C = \frac{3\times \mbox{ number of triangles}}{\mbox{number of connected triples}} \in[0,1].
\end{eqnarray}
In our model of random graphs where vertices are exchangeable, this definition can also be interpreted as the conditional probability that there is an edge between two vertices $j$ and
$k$, given that they have a common neighbor $i$. At the end of this subsection, we will also consider a local clustering coefficient defined for a vertex of degree larger than $3$ and compute the associated clustering coefficient for the graph based on this local measure for the graph $\tGr$.

\paragraph{Computation of the clustering coefficient.}

Note that the number of connected triples in $\tGr$ is simply $\sum_{v}d_v(d_v-1)/2$. On the other hand, for any vertex $v$ in $\tGr$, let $P_v$ be the number of pairs of neighbors of $v$ that share an edge together. More precisely, if $\cN_v$ is the set of neighbors of $v$ (whose cardinality is $|\cN_v|=d_v$), then $P_v$ is the number of pairs $\{w,w'\}\subset\cN_v$, $w\neq w'$, such that $w$ and $w'$ are also neighbors of each other. Thus $3$ times the number of triangles in $\tGr$ is given by $\sum_v P_v$. Hence following (\ref{def:clustering}), we define the clustering coefficient $C^{(n)}$ of the graph $\tGr$ by: 
\[C^{(n)}= \frac{2\cdot \sum_v P_v }{ \sum_v d_v(d_v-1)}\in [0,1].\] 
\begin{Proposition}\label{prop:clustering2}
 We consider the model $\tGr$ for a sequence $\boldd$ satisfying
 Condition \ref{cond}  with probability distribution
$\boldp=\pr$, and clustering parameter $\boldg=\cd$. Then we have for the clustering coefficient of $\tGr$:
\begin{eqnarray*}
C^{(n)} \overset{p}{\longrightarrow} C:= \frac{ \sum_{r \geq 2}
r(r-1)(r-2)\gamma_r p_r}{\sum_{r\geq 2}((r-1)\gamma_r+1)r(r-1)p_r}.
\end{eqnarray*}
\end{Proposition}

The proof is given at the end of the subsection. 
We will explore in more details the implications of Proposition \ref{prop:clustering2} in Section \ref{subs:fix_tilde_p}. Before that we present an alternative definition of clustering.

\paragraph{Another definition of clustering coefficient.}

The local clustering coefficient $C^{(n)}_v$ of a vertex $v$ in a graph quantifies how close the vertex and its neighbors are to being a clique. $C^{(n)}_v$ is defined to be the fraction of pairs of neighbors of $v$ that are also neighbors of each other \cite{ws98}. Using the notations of the previous paragraph, the local clustering coefficient of $v$ is $C^{(n)}_v=P_v\cdot 2/[d_v(d_v-1)]$. Note that this definition only makes sense if $d_v\geq 2$, i.e. when $v$ is not isolated, nor a leaf of the graph. The local clustering of vertices with degree one or zero is taken to be zero.
The clustering coefficient for the whole network $C_2^{(n)}$ is then defined
as the average of the clustering coefficients for each vertex: $C_2^{(n)}=\sum_v C^{(n)}_v /\tn$, where $\tn$ is the number of vertices in the graph, including those of degree one or zero. As observed in \cite{kaiser}, since we take the convention that the local clustering is zero for vertices with degree one or zero, the clustering coefficient in the graph can be very low if the graph contains a lot of such vertices, even if other vertices are highly clustered. We call $C_2^{(n)}$ the biased clustering coefficient and we refer to \cite{kaiser} for more information. In the next proposition, we give the asymptotics for the biased clustering coefficient in $\tGr$ but we will mainly deal with the definition (\ref{def:clustering}) in the rest of the paper. 

\begin{Proposition}\label{prop:clustering}
 We consider the model $\tGr$ for a sequence $\boldd$ satisfying
 Condition \ref{cond}  with probability distribution
$\boldp=\pr$. Then we have for the biased clustering coefficient of $\tGr$:
\begin{eqnarray*}
C_2^{(n)} \overset{p}{\longrightarrow} C_2:= \sum_{r \geq 3}
p_r\frac{\gamma_r}{\tilde{\gamma}} \left(r-2\right),
\end{eqnarray*}
where $\tilde{\gamma}$ is defined in Proposition \ref{prop:deg}.
\end{Proposition}

\begin{proof}[Proof of Propositions \ref{prop:clustering2} and \ref{prop:clustering}]
We recall the following standard result for the random graph $\Gr$:
\begin{Lemma}
\label{lem:clustering}
Let $\bC^{(n)}$ (resp. $\bC^{(n)}_2$) be the (resp. biased) clustering coefficient in $\Gr$. Then we have: $\bC^{(n)} \overset{p}{\longrightarrow} 0$ and $\bC^{(n)}_2 \overset{p}{\longrightarrow} 0$.
\end{Lemma}

We say that a vertex in $\tGr$ has \textit{parent} $i\in\intn$ if it belongs to a clique that replaces the vertex $i$ of $\Gr$ (when $X(i)=1$) or if it is $i$ (when $X(i)=0$). 
We first consider a vertex $v$ in $\tGr$ whose parent $i$ is such that $X(i)=1$. In this case we can directly compute the local clustering coefficient $C^{(n)}_v$. Indeed, vertex $v$ has $d_i-1$ neighbors inside $K$, that are all linked together (which gives $\frac{(d_i-1)(d_i-2)}{2}$ edges in total), and one neighbor $v'$ outside $K$, which is not linked to the other neighbors of $v$ (if it were the case, there would be multiple edges between $i$ and the parent $j$ of $v'$, which is not the case in the simple graph $\Gr$). Hence 
$$
 P_v =\frac{(d_i-1)(d_i-2)}{2} \;\textrm{ and }\;
 C^{(n)}_v = \frac{2 P_v}{d_i(d_i-1)} = \frac{d_i-2}{d_i},
$$
provided that $d_i\geq 2$. If $d_i\in\{0,1\}$, then $C^{(n)}_v=0$. 

We first prove Proposition \ref{prop:clustering}. Since there are $d_i$ such vertices inside a clique, the contribution of clique $K$ in the total clustering $C_2^{(n)}=\sum_v C^{(n)}_v /\tn $ is equal to $d_i C^{(n)}_v/\tn=(d_i-2)/\tn$. This leads to the following:
\begin{eqnarray*}
 \frac{\tn}{n} \; C_2^{(n)}&=& \frac{1}{n} \sum_{d\geq 2} (d-2) B_d + \frac{1}{n} \sum_{i:X(i)=0} C^{(n)}_i
\end{eqnarray*}
where $B_d$ is the number of vertices with degree $d$ that are replaced by a clique, as in the proof of Proposition \ref{prop:deg}. Using that $B_d/n \to_p \gamma_d p_d$, and that $\sum_{i:X(i)=0} C^{(n)}_i /n \to_p 0$ (as a consequence of Lemma \ref{lem:clustering}), we obtain: $\frac{\tn}{n} \; C_2^{(n)} \overset{p}{\longrightarrow} \sum_{d \geq 3}\left(d-2\right)\gamma_d p_d$. Proposition \ref{prop:clustering} follows, applying Proposition \ref{prop:deg}.

The end of the proof for Proposition \ref{prop:clustering2} is similar, and follows from the fact that:
\begin{eqnarray*}
 2\sum_v P_v/n 
 &\to_p& \sum_d d(d-1)(d-2) \gamma_d p_d, \\
\sum_v d_v(d_v-1)/n 
 & \to_p& \sum_d d(d-1)[ d\gamma_d +(1- \gamma_d) ] p_d .
\end{eqnarray*}
\end{proof}

\subsection{Tunable clustering coefficient with fixed degree distribution}
\label{subs:fix_tilde_p}

In this subsection, we show how to use our model in order to generate graphs with a given degree distribution and clustering. This construction will allow us to compare graphs with a given degree distribution but with various clustering coefficients and to see the impact of clustering on the epidemic.
This analysis is not possible with the random intersection graphs studied in \cite{britton-2007}. Indeed, once the clustering coefficient and the mean degree in the graph is fixed in \cite{britton-2007}, the degree distribution is completely determined and has to be compound Poisson. In particular, the variance (as all higher moments) of the degree distribution is also fixed. 

Our model has a lot more freedom in term of graphs that we can generate but still has one limitation: for a given degree distribution, there is a constraint on the maximal value of the clustering coefficient for our model. As a simple example, note that our model is not able to generate $2$-regular graphs with positive clustering coefficient. In order to provide a graph with a given asymptotic degree distribution $\boldsymbol{\tp}$ and a positive clustering coefficient, we need the following assumptions on $\boldsymbol{\tp}$:
\begin{Condition}
 \label{cond:tilde}
We assume that the probability distribution $\boldsymbol{\tp}$ satisfies:
\begin{itemize}
\item[(i)] $\sum_r r^2 \tp_r <\infty$;
 \item[(ii)] $\sum_{r\geq 3} \tp_r>0$;
\item[(iii)] $\tp_0=0$.
\end{itemize}
\end{Condition}
Under these conditions, we have the following proposition:
\begin{Proposition}\label{prop:algo}
Let $\boldsymbol{\tp}=(\tp_r)_{r\geq 0}$ be a probability distribution satisfying Condition \ref{cond:tilde}.
We define the maximal clustering coefficient as
\begin{eqnarray} \label{eq:Cmax}
 C^{\max}:=1-\frac{2\sum_{r\geq 2}(r-1)\tilde{p}_r}{\sum_{r\geq 2}r(r-1)\tilde{p}_r} .
\end{eqnarray}
Then for any value $0\leq C\leq C^{\max}$, there exists a sequence $\boldd$ satisfying Condition \ref{cond} with probability distribution $\boldp=\pr$
and a value of $\gamma\in [0,1]$ such that the model $\tGrg$ has asymptotic degree distribution $\boldsymbol{\tp}$ and asymptotic clustering coefficient $C$. 

More precisely, $\gamma$ is the solution of the following equation:
\begin{eqnarray} \label{eq:C}
 C=C(\gamma):=\frac{\sum_{r\geq 3} r(r-1)(r-2) \frac{\gamma}{(r-1)\gamma+1}\tilde{p}_r}{\sum_{r\geq 2}r(r-1)\tilde{p}_r}.
\end{eqnarray}
Let  $F(\gamma') := \sum_{r\geq 1} \frac{r}{(r-1)\gamma'+1}\tilde{p}_r$ for all $\gamma'\in[0,1]$, and set
\begin{equation}\label{eq:lambda}
 \lambda := 
\begin{cases} \frac{F(\gamma)(1-\gamma)}{1-\gamma F(\gamma)} & \text{if $\gamma \neq 1$,}
\\
\frac{1}{\sum_{r\geq 1} \tp_r /r} &\text{if $\gamma=1$.}
\end{cases}
\end{equation}
Then we can define $\boldp$ as:
\begin{eqnarray} \label{eq:p}
p_r := \frac{\tilde{p}_r[(\lambda-1)\gamma +1]}{(r-1)\gamma+1}\;\textrm{for all } r\geq 1, \; \textrm{and} \; p_0:=0.
\end{eqnarray}
\end{Proposition}

\begin{proof}
 We first show that $\gamma$, $\lambda$ and $\boldp$ are well-defined, and that $\boldp$ is a probability distribution that satisfies Condition \ref{cond}.

The solution $\gamma$ of equation \eqref{eq:C} exists and is unique. Indeed we have the following derivative for $C(\gamma)$: $$\frac{dC}{d\gamma}(\gamma) = \frac{\sum_{r\geq 3} r(r-1)(r-2) \frac{1}{[(r-1)\gamma+1]^2}\tilde{p}_r}{\sum_{r\geq 2}r(r-1)\tilde{p}_r},$$
which is positive since $\sum_{r\geq 3} \tp_r>0$. Hence the clustering coefficient $C(\gamma)$ is an increasing function of $\gamma$, taking values between $0$ and $ C(1)= C^{\max}$. Since $C\in[0,C^{\max}]$, there exists a unique $\gamma\in[0,1]$ such that $C(\gamma)=C$.

The constant $\lambda$ is well-defined: if $\gamma<1$, then for all $r\geq 1$, $(r-1)\gamma+1>r\gamma$, which leads to $F(\gamma)<1/\gamma$, \ie $1-\gamma F(\gamma)>0$. 

In addition, we have that $\sum_r rp_r=\lambda$. Indeed, if $\gamma=1$, then $rp_r=\lambda \tp_r$, and summing over $r\geq 0$ gives the result. If $\gamma\neq 1$, we have that $\sum_r rp_r=[(\lambda-1)\gamma +1] F(\gamma)=\lambda$, where the first equality comes from \eqref{eq:p} and the second one from \eqref{eq:lambda}.

We can easily verify that $\boldp$ is a probability distribution: for all $r\geq 0$, we have that $\left((r-1)\gamma+1\right)p_r=\tp_r[(\lambda-1)\gamma +1]$ (due to \eqref{eq:p} and the fact that $\tp_0=0$). Summing over all $r\geq 0$ and using the fact that $\sum_r rp_r=\lambda$ finally gives that $\sum_r p_r=1$. 

We know consider the graph $\tGrg$, with $\boldd$ given by the following. For each $n$, let $\boldd$ such that $\left|\{i:d_i=r\}\right|= \lfloor n p_{r}\rfloor$ for all $r\geq 0$. In addition we adjust the value of $d_n$ (for instance) such that $\sum_i d_i$ is even. Then $\sum_i d_i^3=\sum_r r^3 \lfloor n p_r\rfloor =O(n)$ due to Condition \ref{cond:tilde}-\textit{(ii)} and equation \eqref{eq:p}. Hence Condition \ref{cond} is satisfied by $\boldd$.

We can verify that $\tGrg$ has asymptotic degree distribution $\boldsymbol{\tp}$, using Proposition \ref{prop:deg} and equation \eqref{eq:p}, and asymptotic clustering coefficient $C$, using Proposition \ref{prop:clustering2}, equations \eqref{eq:C} and \eqref{eq:p}.
\end{proof}

A similar result can be proved with the biased clustering coefficient. In that case, $C(\gamma)$ is replaced by $C_2(\gamma):=\sum_{r\geq 3} \frac{(r-2) \gamma}{r\gamma+1-\gamma} \tilde{p}_r$ and $C_2^{\max}:=\sum_{r\geq 3}\frac{r-2}{r}\tilde{p}_r$. 
We can see on Figure \ref{fig:Clust} that the interval of reachable clustering values is larger for the first notion of clustering. This illustration uses a power law degree distribution with exponential cutoff for the distribution $\boldsymbol{\tp}$: there exists a power $\tau>0$ and a cutoff $\kappa>0$ such that, for all $r\geq 1$, $\tp_r=c(\tau,\kappa) \cdot r^{-\tau} e^{-r/\kappa}$, where $c(\tau,\kappa)=1/(\sum_s s^{-\tau} e^{-s/\kappa}) $ is a normalizing constant. This cutoff $\kappa$ allows Condition \ref{cond:tilde} to be satisfied for any power $\tau>0$: in all figures, we will take $\kappa=50$. In order to increase the mean degree of the graph in Figure \ref{fig:Clust}, we decrease the power $\tau$.

\begin{figure}
\centering
\epsfig{file=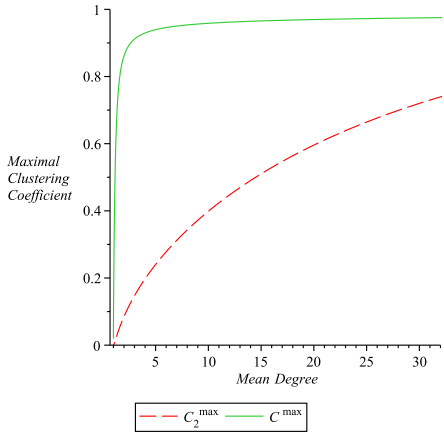, width=6cm}
\caption{Reachable clustering values for both notions of clustering coefficient, with respect to the mean degree $\tl$ in $\tGr$ (when the graph $\tGr$ has power law degree distribution $\tp_r \propto r^{-\tau} e^{-r/50}$)}
\label{fig:Clust}
\end{figure}

\section{Diffusion threshold for random graphs with clustering}\label{sec:diff}

\subsection{Diffusion model}\label{subs:def_diff}

In this section, we study a simple diffusion model depending on a
single parameter $\pi\in [0,1]$. For a given graph $G$, the dynamics of the diffusion is as follows: some
set of nodes $S$ starts out being active; all other
nodes are inactive. When a node becomes active, each of her neighbors becomes active with probability $\pi$ independently from each other. 
The final state of the diffusion can also be
described in term of a bond percolation process in the graph $G$. 
Randomly delete each
edge with probability $1-\pi$ independently of all other edges. Denote
by $G_\pi$ the resulting graph. Then
any node in $S$ will activate all nodes in its connected component in
$G_\pi$. 

\subsection{Phase transition for the diffusion with a single activation} \label{subs:diff}

In this subsection we consider diffusion starting from one active node and all other nodes being inactive, and we derive conditions under which a single starting active node can activate a large fraction of the
population in $G=\tGr$. This problem corresponds to the existence of a
'giant component' in the random graph obtained after bond percolation.

In order to state our result, we first need to recall some basic results about random graphs with small
order. For $d\in \NN$, let $K_d$ be the complete graph on
$d$ vertices denoted $\{1,\dots , d\}$, with $d(d-1)/2$ edges. For $\pi\in [0,1]$, we denote
by $K_d(\pi)$ the random graph obtained from $K_d$ after bond
percolation with parameter $\pi$, i.e. each edge of $K_d$ is kept
independently of the others with probability $\pi$, otherwise it is
removed.

We need to compute the probability that the component in $K_d(\pi)$
containing vertex $1$ has $k$ vertices, denoted by $f(d,k,\pi)$. Note
that $f(d,d,\pi)$ is simply the probability that $K_d(\pi)$ is
connected and has been computed in \cite{gil59}. Indeed simple
computations show that we have the recurrence relation
\begin{eqnarray}
\nonumber f(d,d,\pi) &=& 1-\sum_{k=1}^{d-1} {d-1\choose k-1}
f(k,k,\pi)(1-\pi)^{k(d-k)},\\
\label{eq:deff}f(d,k,\pi) &=& {d-1\choose k-1}f(k,k,\pi)(1-\pi)^{k(d-k)},
\end{eqnarray}
for any $k\leq d$.

We now define for $d\in \NN$ and $\pi\in [0,1]$, the random variable $\cK(d,\pi,\boldg)$ by
\begin{eqnarray*}
\Pb\left( \cK(d,\pi,\boldg)=k\right)= (1-\gamma_d) \Indb(d=k) + \gamma_d f(d,k,\pi),
\end{eqnarray*}
where $f$ is defined in (\ref{eq:deff}). In words, $\cK(d,\pi,\boldg)$ is equal to $d$ with probability $1-\gamma_d$ and to the size of the component in $K_d(\pi)$ containing $1$ with the remaining probability. 

In addition, set for all $k\geq 1$:
\begin{eqnarray}
 \varrho_k &:=& p_k(1-\gamma_k)+\sum_{d\geq k} \frac{d}{k} f(d,k,\pi) p_d \gamma_d, \label{def:rhok} \\
 \varrho &:=& \sum_{\ell} \varrho_{\ell}, \label{def:rho} \\ 
 \mu &:=& \sum_{\ell} \ell \varrho_{\ell} / \varrho, \label{def:mu} \\
 \sigma_k &:=& p_k(1-\gamma_k)+\sum_{d\geq k} d \, f(d,k,\pi) p_d \gamma_d. \label{def:sigma} 
\end{eqnarray}
Using the notation $\tgamma$ defined in Proposition \ref{prop:deg}, we also define (omitting the dependence on $\boldp$ and $\boldg$):
\begin{eqnarray*}
L(z)&:=&\sum_{s} \frac{\sigma_s}{\tgamma} \left[1-(1-\pi+\pi z)^s\right] \nonumber \\
 h(z)&:=&\sum_{s} s \frac{\varrho_s}{\varrho} (1-\pi+\pi z )^s \nonumber  \\
 \label{eq:zetaDiff} \zeta &:=&\sup\{\, z\in[0,1):\mu z(1-\pi+\pi z)=h(z)\}.
\end{eqnarray*}

For a graph $G=(V,E)$ and a parameter $\pi\in [0,1]$, we denote by $C^{b}(\pi)$ the size of the largest component in the bond percolated graph $G_\pi$.

\begin{theorem}\label{th:diff}
Consider the random graph $G=\tGr$ for a sequence $\boldd$ satisfying
Condition \ref{cond} with probability distribution
$\boldp=\pr$, and clustering parameter $\boldg=\cd$.
Let $D^*$ be a random variable with distribution $p^*_r$ given by $p^*_{r-1} =
\frac{rp_r}{\lambda}$ for all $r\geq 1$.
We define $\pi_c$ as the solution of the equation:
\begin{eqnarray*}
\pi \Eb\left[ \cK(D^*+1,\pi,\boldg)-1\right]=1.
\end{eqnarray*}
\begin{itemize}
\item[(i)] if $\pi>\pi_c$, we have (with the notations above) that $\zeta\in(0,1)$. In addition the asymptotic size of the largest component of the percolated graph $G_{\pi}$ obtained from $\tGr$ is:
\begin{eqnarray*}
 C^{b}(\pi)/\tn &\overset{p}{\longrightarrow}& L(\zeta) >0.
\end{eqnarray*}
\item[(ii)] if $\pi<\pi_c$, we have $C^{b}(\pi)=o_p(\tn)$.
\end{itemize}
\end{theorem}

We can guess the value of the diffusion threshold $\pi_c$ using a branching process approximation (Appendix \ref{app:BP}).

Note that in the particular case where $\gamma_r=0$ for all $r$, we have
$\cK(d,\pi,0)=d$ so that we get $\pi_c=\frac{\Eb[D]}{\Eb[D(D-1)]}$ where
$D$ is the typical degree in the random graph and our result reduces to a
standard result in the random graphs literature (see Theorem 3.9 in
\cite{ja09}).

The case where a positive fraction of individuals belong to $S$ (not only a single node) is discussed in Subsection \ref{subs:diffA}. Now we study the effect of clustering on the diffusion with a single activation.

\subsection{Effect of clustering on the diffusion for regular graphs}

In this paragraph, we consider $d$-regular graphs ($d\geq 3$), so that $p_r=\mathds{1}(r=d)$ for all $r\geq0$. In this case, adding cliques does not change the asymptotic degree distribution in the graph $\tGr$, and $\tp_r=\mathds{1}(r=d)=p_r$. In addition, we assume that $\gamma_r=\gamma$ for all $r\geq 0$ (each vertex in $\Gr$ is replaced by a clique with probability $\gamma$). We are interested in the effect of clustering on the diffusion threshold $\pi_c$ on the one hand, and its effect on the epidemic size on the other hand. 

We have the following result for the diffusion threshold in regular graphs:
\begin{Proposition}
 Let $d\geq 3$. We consider the asymptotic degree distribution $\boldsymbol{\tp}=(\Indb_{r=d})_{r\geq 0}$. Let $0\leq C^{(1)}<C^{(2)}\leq C^{\max}=1-2/d$. For each $j=1,2$, let $(\boldd_j,\boldp_j,\gamma_j)$ be chosen according to Proposition \ref{prop:algo}, such that $\tG^{(j)}=\tG(n,\boldd_j,\gamma_j)$ has asymptotic degree sequence $\boldsymbol{\tp}$ and asymptotic clustering coefficient $C^{(j)}$. Let $\pi_c^{(j)}$ be the diffusion threshold defined in Theorem \ref{th:diff} for the random graph $\tG^{(j)}$, $j=1,2$. Then we have:
\begin{eqnarray*}
 \pi_c^{(1)}\leq \pi_c^{(2)}
\end{eqnarray*}
In our graph model, the diffusion threshold for a random $d$-regular graph increases as the clustering coefficient increases.
\end{Proposition}

\begin{proof}
Both notions of clustering coefficient considered in \ref{subs:clustering} are the same for random $d$-regular graphs, and we have that: $C(\gamma)=\frac{d-2}{d-1+1/\gamma}$. In particular, $C^{(1)}<C^{(2)}$ implies that $\gamma_1< \gamma_2$.

According to Theorem \ref{th:diff}, the diffusion threshold $\pi_c^{(j)}$, $j=1,2$, is the solution of the following equation:
\begin{eqnarray}\label{eq:pi_d_reg}
\pi \Eb\left[ \cK(d,\pi,\gamma_j)-1\right]=1.
\end{eqnarray}
Using the definition of $\cK$, equation \eqref{eq:pi_d_reg} becomes $\pi=F(\pi,\gamma_j)$, where
\[
F: \left\{\begin{array}{rcl}
    [0,1]^2 & \longrightarrow & [0,1] \\
    (\pi,\gamma) & \mapsto & \frac{1}{d-1+\gamma\left(\sum_{k=1}^d k f(d,k,\pi)-d\right)}.
    \end{array} \right.
\]

Let $\pi\in[0,1]$. Then $F(\pi,\gamma_1)\leq F(\pi,\gamma_2)$ since $\gamma_1<\gamma_2$ and $\sum_{k=1}^d k f(d,k,\pi)\leq d$. Indeed, $f(d,k,\pi)$, $k\leq d$, is the probability that the connected component of a vertex inside $K_d(\pi)$ has $k$ vertices in $K_d(\pi)$. Hence $\sum_{k=1}^d k f(d,k,\pi)$ is the mean size of that component inside $K_d(\pi)$, so that $\sum_{k=1}^d k f(d,k,\pi)\leq d$.

Therefore the curve $\Gamma_2$ of $\pi \mapsto F(\pi,\gamma_2)$ is above the curve $\Gamma_1$ of $\pi \mapsto F(\pi,\gamma_1)$. Let $A_1$ (resp. $A_2$) be the intersection between $\Gamma_1$ (resp. $\Gamma_2$) and the first bisector. Both functions are continuous, thus the first coordinate of $A_1$ is less than or equal to the one of $A_2$, that is to say $\pi_c^{(1)}\leq \pi_c^{(2)}$.
\end{proof}

\begin{figure}
\centering
\epsfig{file=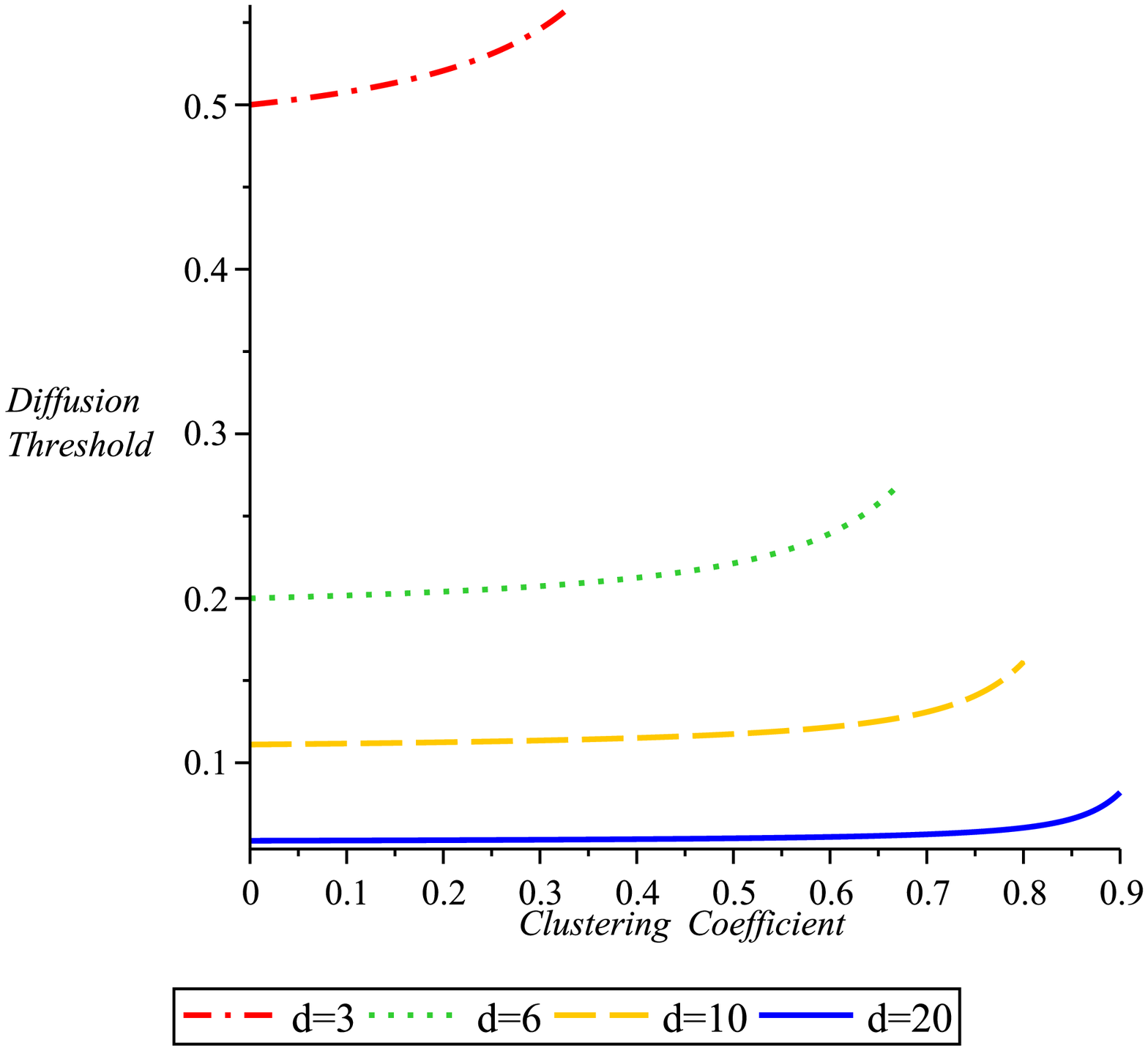, width=6.3cm}
\epsfig{file=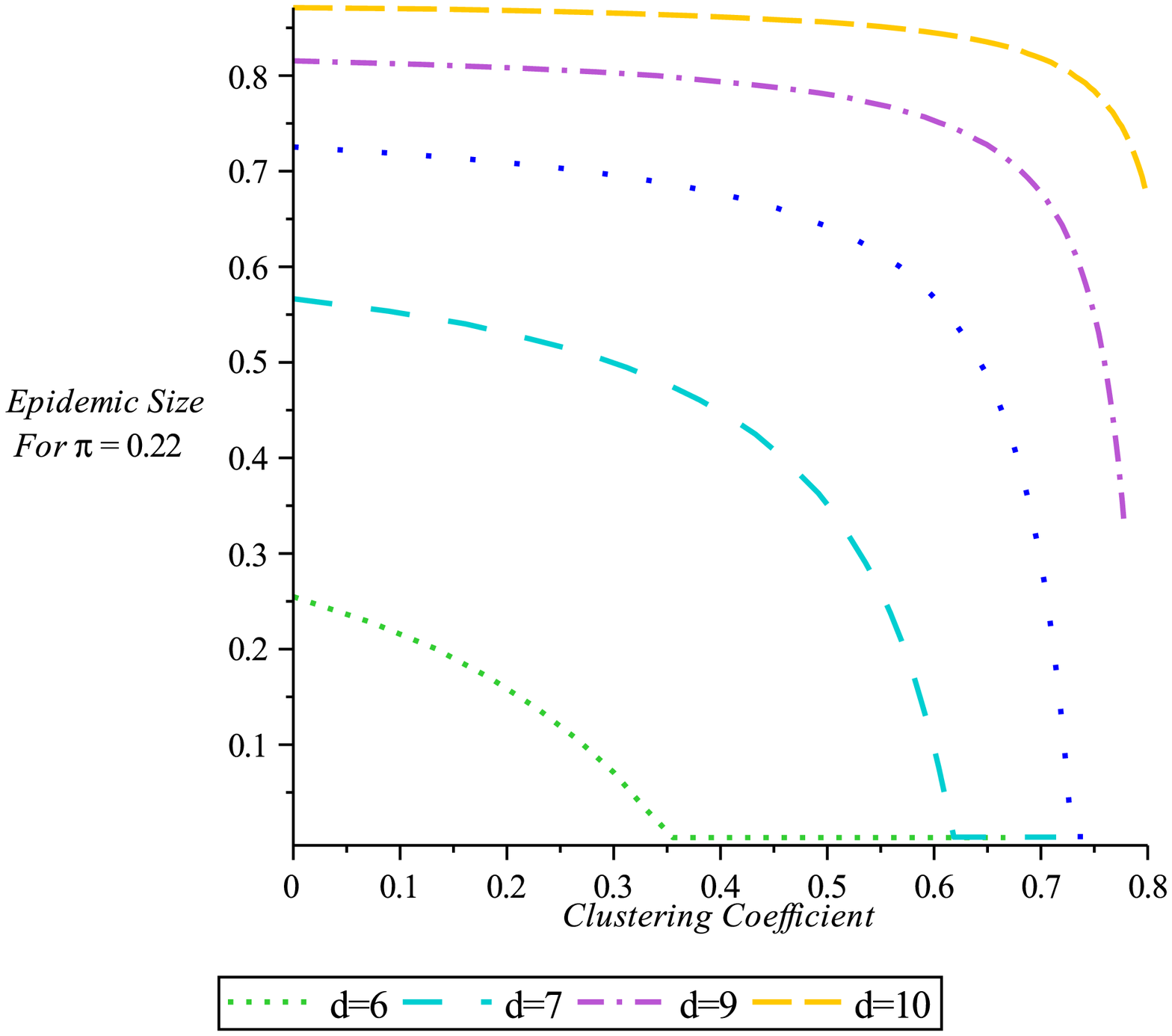, width=6.3cm}
\caption{On the left: Evolution of the diffusion threshold with respect to the clustering coefficient for $d$-regular graphs. On the right: Evolution of the epidemic size with respect to the clustering coefficient for $d$-regular graphs (with infection probability $\pi=0.22$).}
\label{fig:Diff2D}
\end{figure}

More precisely, Figure \ref{fig:Diff2D} (on the left) shows how the diffusion threshold increases with the clustering coefficient, for different values of $d$: in other words, clustering decreases the range of $\pi$ ($\pi\in(\pi_c;1]$) for which a single individual can turn a positive proportion of the population into infected individuals.

In addition, the epidemic size also decreases with the clustering: in Figure \ref{fig:Diff2D} (on the right), we plot the ratio of the largest connected component in the percolated graph over the whole population. When the starting infected individual is a vertex chosen uniformly at random, this ratio also corresponds to the probability of explosion. Hence, as the clustering increases, it 'inhibits' the diffusion process. These results are in accordance with \cite{GleesonMH:clustering}.

These results are intuitive (for $d$-regular graphs) in the sense that the removal of edges inside cliques can stop the diffusion \textit{inside} a clique in the graph $\tGr$, while this phenomenon does not occur in the original graph $\Gr$.

\subsection{Effect of clustering on the diffusion for graphs with power law degree distribution}

\begin{figure}
\centering
\epsfig{file=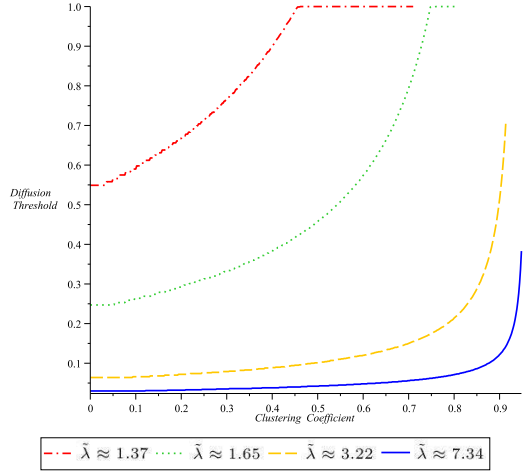, width=7cm}
\caption{Evolution of the diffusion threshold with respect to the clustering coefficient in a graph with mean degree $\tl$, with respect to the clustering coefficient $C$ (for a fixed power law degree distribution). }
\label{fig:DiffPL}
\end{figure}

In Figure \ref{fig:DiffPL}, we consider a power law degree distribution with exponential cutoff: $\tp_r \propto r^{-\tau} e^{-r/50}$, with parameters $\tau=2.9$, $\tau=2.5$, $\tau=1.81$, $\tau=1.3$, so that the mean degree is respectively $\tl\approx 1.37$, $\tl\approx 1.65$, $\tl\approx 3.22$, $\tl\approx 7.3$. We plot the diffusion threshold $\pi_c$ for the graph given by Proposition \ref{prop:algo}, when the degree distribution is $\boldsymbol{\tp}$ and the clustering coefficient varies from $0$ to $C^{\max}$. We observe the same phenomenon as the one we proved for $d$-regular graphs, \ie clustering decreases the range of $\pi$ ($\pi\in(\pi_c;1]$) for which a single individual can turn a positive proportion of the population into infected individuals.


\subsection{Phase transition for the diffusion with degree based activation} \label{subs:diffA}

In this subsection, we allow a positive fraction of nodes to be active at the beginning of the diffusion process. More precisely, on a given graph $G$, the set $S$ of initial active nodes is random, and each node of degree $d$ in $G$ belongs to $S$ with some probability $\alpha_d> 0$, independently for each node. We set $\balpha=(\alpha_d)_{d\geq 0}$. 


Using the notation $\tgamma$ defined in Proposition \ref{prop:deg}, and definitions \eqref{def:rhok} to \eqref{def:mu}, we define (omitting the dependence on $\balpha$, $\boldp$ and $\boldg$):

\begin{eqnarray}
L(z)&:=&\sum_s \frac{(1-\gamma_s)p_s}{\tgamma}\left[1-(1-\alpha_s)(1-\pi+\pi \zeta)^s\right] \nonumber \\
&& + \sum_{d\geq s} d f(d,s,\pi) \frac{\gamma_d p_d}{\tgamma} \left[1-(1-\alpha_d)^s(1-\pi+\pi \zeta)^s\right] \nonumber \\
 h(z)&:=& \sum_s s(1-\gamma_s)p_s\left[1-(1-\alpha_s)(1-\pi+\pi \zeta)^s\right]/\varrho \nonumber \\
&& + \sum_{d\geq s} d f(d,s,\pi) \gamma_d p_d \left[1-(1-\alpha_d)^s(1-\pi+\pi \zeta)^s\right]/\varrho \nonumber \\
 \label{eq:zetaDiffA} \zeta &:=&\sup\{\, z\in[0,1):\mu z(1-\pi+\pi z)=h(z)\}.
\end{eqnarray}

\begin{theorem}\label{th:diffA}
Consider the random graph $G=\tGr$ for a sequence $\boldd$ satisfying
Condition \ref{cond} with probability distribution
$\boldp=\pr$, and clustering parameter $\boldg=\cd$. We are given an activation set $S$ drawn according to the distribution $\balpha$.
Then we have, for the diffusion model defined in \ref{subs:def_diff}: if $\zeta=0$, or if $\zeta\in(0,1]$, and further $\zeta$ is such that there exists $\eps>0$ with $\lambda z(1-\pi+\pi z)<h(z)$ for $z\in (\zeta-\eps,\zeta)$, then we have 
that the size $C^{b}(\pi,\balpha)$ of the active nodes at the end of the diffusion verifies:
\begin{eqnarray*}
 C^{b}(\pi,\balpha)/\tn &\overset{p}{\longrightarrow}& L(\zeta).
\end{eqnarray*}
\end{theorem}

Heuristically, taking $\alpha_s=0$ for all $s$ in the definitions of the previous theorem allows to recover the result of Theorem \ref{th:diff}.

\section{Symmetric threshold model for random graphs with clustering}\label{sec:cont}

\subsection{Symmetric threshold model} \label{subs:def_cont}

We now describe the symmetric threshold model on a finite graph $G=(V,E)$, with given thresholds $k(v)$, for $v\in V$. 
The progressive dynamics of the epidemic on the finite graph $G$ operates as follows: some
set of nodes $S$ starts out being active; all other
nodes are inactive.
Time operates in discrete steps $t=1,2,3,\dots$. At
a given time $t$, any inactive node $v$ becomes
active if its number of active neighbors is at least $k(v)+1$.
This in turn may cause other nodes to become active.
It is easy to see that the final set of active nodes (after $n$ time
steps if the network is of size $n$) only depends on the initial set
$S$ (and not on the order of the activations) and can be obtained
as follows: set $Y_v=\Indb(v\in S)$ for all $v$.
Then as long as there exists $v$ such that
$\sum_{w\sim v}Y_w> k(v)$, set $Y_v=1$, where $w\sim v$
  means that $v$ and $w$ share an edge in $G$.
When this algorithm finishes, the final state of node $v$ is
represented by $Y_v$: $Y_v=1$ if node $v$ is active and $Y_v=0$
otherwise. In this paper, we do not analyze the dynamics of the epidemics and concentrate on the final state only.

We allow the threshold $k(v)$ of a node $v$ to be a random variable with distribution depending on the degree of $v$, and such that thresholds are independent among nodes. More precisely, for each $s\geq 0$, let $(t_{s\ell})_{0\leq \ell\leq s}$ be a probability distribution. We draw independent thresholds $k(i)$, for $i\in V$. Knowing that the degree $d_i$ of node $i$ is $s$, threshold $k(i)$ is drawn according to the conditional probability distribution $(t_{s\ell})_{0\leq \ell\leq s}$: $\Pb(k(i)=\ell|d_i=s)=t_{s\ell}$. We say that random thresholds $\boldk=(k(i))_{i\in V}$ are drawn according to $\boldt={(t_{s\ell})}_{s,\ell}$.

To simplify, we define an adaptation of the symmetric threshold model for the random graph $\tGr$: we draw random thresholds $\boldk$ for each vertex $i$ in the original graph $\Gr$. When a vertex $i$ of $\Gr$ is replaced by a clique in $\tGr$, we associate to each vertex inside the clique the original threshold $k(i)$, so that vertices inside a clique have the same threshold (also referred to as the ``threshold of the clique''). We still denote by $k(v)$ the threshold of a vertex $v$ in $\tGr$.




\subsection{Phase transition for the symmetric threshold model with a single activation}

\label{subs:contagion}

We show that there is a phase transition for a single active node to turn a positive fraction of the population into active nodes.
For a graph $G=(V,E)$ and thresholds $\boldk=(k(v))_{v\in V}$, we consider the largest connected
component of the induced subgraph in which we keep only vertices of threshold zero. We call the vertices in this component pivotal
players: if only one pivotal player becomes active then the
whole set of pivotal players will eventually become active. In particular, if the set of pivotal players is large (\ie of order $\Theta(|V|)$, as $|V|\to\infty$), then a single pivotal player $u$ can trigger a global cascade (\ie the number of active nodes at the end of the epidemic starting from $u$ is of order $\Theta(|V|)$).  

We consider the random graph $G=\tGr$, and random thresholds drawn according to the distribution $\boldt$. For a node $v$, we denote by $C(v,\boldt)$ the final number of active vertices, when the initial state consists of only $v$ active and all
other nodes are inactive. Informally, we say that $C(v,\boldt)$ is the
size of the cascade induced by node $v$; if $C(v,\boldt)=\Theta_p(\tn)$, we say that node $v$ can trigger a global cascade.

Using the notation $\tgamma$ defined in Proposition \ref{prop:deg} and the binomial probabilities $b_{sr}(p)$ defined at the end of Section \ref{sec:intro}, we set (omitting the dependence on $\boldt$, $\boldp$ and $\boldg$):
\begin{eqnarray}
L(z)&:=&\sum_{s} \frac{\left[s \gamma_s+(1-\gamma_s)\right] p_s}{\tgamma} t_{s0}(1-z^{s}) \nonumber  \\ && +\sum_{s} \frac{(1-\gamma_s) p_s}{\tgamma} \left(1-t_{s0}-\sum_{\ell\neq 0}t_{s\ell}\sum_{r\geq s-\ell}b_{sr}(\zeta)\right), \nonumber  \\
 h(z)&:=&\sum_{s} sp_s \left[t_{s0}z^{s}+\gamma_s(1-t_{s0})z\right] \nonumber  \\ && +\sum_{s}p_s (1-\gamma_s)\sum_{s\geq \ell\neq 0}t_{s\ell}\sum_{r\geq s-\ell} r b_{sr}(z),\nonumber \\
 \label{eq:zetaCont} \zeta &:=&\sup\{\, z\in[0,1):\lambda z^2=h(z)\}.
\end{eqnarray}

\begin{theorem}
\label{th:contagion}
Consider the random graph $\tGr$ for a sequence $\boldd$ satisfying
Condition \ref{cond} with probability distribution $\boldp=\pr$, and clustering parameter $\boldg=\cd$. Let $\boldt$ be a family of probability distributions, and $\boldk$ random thresholds drawn according to $\boldt$ in the original graph $\Gr$ (\ie if $i$ is a vertex in $\Gr$ replaced by a clique in $\tGr$, then all vertices in the clique have the same threshold $k(i)$).
We call the following condition the cascade condition:
\begin{eqnarray}\label{eq:casc}
\sum_{r} r(r-1) p_r t_{r0}> \sum_r  rp_r.
\end{eqnarray}

Let $\tcP^{(n)}$ be the set of pivotal players in $\tGr$.  
\begin{itemize}
 \item[(i)] If the cascade condition \eqref{eq:casc} is satisfied, then there is a unique $\xi\in(0,1)$ such that
\begin{eqnarray}\label{eq:xi_cont}
 \sum_{d} d p_d t_{d0} (1-\xi^{d-1}) = \lambda (1-\xi)
\end{eqnarray}
and we have:
\begin{eqnarray}\label{eq:pivotal_players}
\frac{|\tcP^{(n)}|}{\tn} \overset{p}{\longrightarrow} \sum_{d} \frac{ \left[d \gamma_d+(1-\gamma_d)\right] p_d t_{d0}}{\tgamma} (1-\xi^{d})>0,
\end{eqnarray}
where $\tgamma$ is defined in Proposition \ref{prop:deg}. Moreover, for any $u\in\tcP^{(n)}$, we have whp
\begin{eqnarray}\label{eq:pivotal_players_1}
 \liminf \frac{C(u,\boldt)}{\tn} \geq L(\zeta)>0
\end{eqnarray}
where $\zeta$ is defined by \eqref{eq:zetaCont}. If in addition $\zeta=0$ or $\zeta$ is such that there exists $\eps>0$ with $\lambda z^2<h(z)$ for $z\in (\zeta-\eps,\zeta)$, then we have for any $u\in\tcP^{(n)}$:
\begin{eqnarray}
\label{eq:pivotal_players_2} C(u,\boldt)/\tn &\overset{p}{\longrightarrow}& L(\zeta).
\end{eqnarray}
 \item[(ii)] If $\sum_{r} r(r-1) p_r t_{r0}< \sum_r  rp_r$, for a uniformly chosen player $u$, we have
   $C(u,\boldt)=o_p(\tn)$. The same result holds if $o(n)$ players are chosen
   uniformly at random.
\end{itemize}
\end{theorem}

When $\gamma_r=0$ for all $r\geq 0$, Theorem \ref{th:contagion} corresponds to the result of
\cite{lel:diff}. When we add cliques in the graph, the effect on the epidemic can be described by the following lemma. 

\begin{Lemma}
 \label{lem:cascade}
We consider a clique in $\tGr$ where all vertices are inactive, and at least one of them has a neighbor outside the clique which is active. If the threshold $k$ of the (vertices in the) clique is zero, then the epidemic will propagate to the whole clique. On the contrary, if $k$ is positive, then the clique cannot become active, even if all neighbors outside are active.
\end{Lemma}

Indeed, if $k=0$, each vertex in the clique needs only one active neighbor to become active. If $k>0$, each vertex in the clique needs at least two active neighbors to become active. Yet each vertex of the clique has only one (active) neighbor outside, other neighbors being (inactive ones) inside the clique.

Hence a clique with positive threshold in which all vertices are initially inactive will always stops the epidemic. This simple observation allows to make a comparison between the epidemic in the original graph $\Gr$ and the epidemic in the graph $\tGr$ with additional cliques: since cliques have a tendency to stop epidemic, it is also easy to see that if there is no global cascade in $\Gr$, then there is no one in $\tGr$ (more details are given in the proof, \ref{subs:pf_cont}).

The fact that the converse is also true is more remarkable, since the cliques with positive threshold stop the epidemic. In fact, those cliques will reduce the \textit{size} of the cascade, but they have no impact on the fact that cascade is possible or not. Indeed it is shown in \cite{lel:diff} that a global cascade is possible in the graph $\Gr$ if and only if the set of pivotal players (in $\Gr$) is large. Note that one direction of this equivalence is easy, and holds for any graph (in particular, this is still true for $\tGr$): if the set of pivotal players is large, any pivotal player which is initially active can trigger a global cascade, as explained at the beginning of the subsection. Let us assume now that there is a global cascade in $\Gr$. Using the equivalence shown in \cite{lel:diff}, the set of pivotal players (in $\Gr$) is large. This implies that the set of pivotal players in the graph $\tGr$ (with additional cliques) is also large (details are given in the proof, \ref{subs:pf_cont}), so that there is a cascade in the graph $\tGr$.

Hence there is a global cascade in $\tGr$ if and only if there is one in the original graph $\Gr$. This explains why the cascade condition \eqref{eq:casc} only depends on the original distribution $\boldp$ and threshold distribution $\boldt$ (and not on $\boldg$). Yet these two graphs ($\Gr$ and $\tGr$) have not the same asymptotic degree distribution. 
What is interesting now is to compare two graphs that have the same asymptotic degree distribution $\boldsymbol{\tp}=(\tp_r)_{r\geq 0}$, but different clustering coefficients.

\subsection{Effect of clustering on the contagion threshold} \label{subs:contagion_clust}

We use our results to highlight the effect of clustering for the game-theoretic contagion model proposed by Blume \cite{bl95} and Morris \cite{mor}. 
Consider a graph $G$ in which the nodes are the individuals in the population and there is an edge $(i,j)$ if $i$ and $j$ can interact with each other. Each node has a
choice between two possible actions labeled $A$ and $B$. On each
edge $(i,j)$, there is an incentive for $i$ and $j$ to have their
actions match, which is modeled as the following coordination game
parametrized by a real number $q\in (0,1)$:
if $i$ and $j$ choose $A$ (resp. $B$), they each receive a payoff of
$q$ (resp. $(1-q)$); if they choose opposite actions, then they
receive a payoff of $0$.
Then the total payoff of a player is the sum of the payoffs with each
of her neighbors. If the degree of node $i$ is $d_i$ and $S_i^B$ is
 the number of its neighbors playing $B$, then the payoff to $i$ from
choosing $A$ is $q(d_i-S_i^B)$ while the payoff from choosing $B$ is
$(1-q)S^B_i$. Hence, in a best-response dynamic, $i$ should adopt $B$ if $S_i^B>qd_i$ and $A$ if
$S_i^B\leq qd_i$.
A number of qualitative insights can be derived from such
a model even at this level of simplicity \cite{klein}. Specifically,
consider a network where all nodes initially play $A$.
If a small number of nodes are forced to adopt strategy $B$ (the seed) and we apply
best-response updates to other nodes in the network, then these nodes will be
repeatedly applying the following rule: switch to $B$ if enough of
your neighbors have already adopted $B$. There can be a cascading
sequence of nodes switching to $B$ such that a network-wide
equilibrium is reached in the limit. Note that the dynamics of the contagion process is deterministic once the seed is fixed as opposed to the diffusion process.
The contagion process is a particular case of the symmetric threshold model. Indeed the threshold distribution is given by:  $t_{s\ell}=\Indb(\lfloor qs \rfloor =\ell)$ for all $0\leq \ell\leq s$. The cascade condition \eqref{eq:casc} is satisfied if and only if the parameter $q$ of the contagion is greater than the contagion threshold 
\begin{eqnarray} \label{eq:qc}
q_c:=\sup \left\{q:\sum_{r<q^{-1}} r(r-1) p_r > \sum_r  rp_r\right\}. 
\end{eqnarray}


We restrict ourselves to the case where $\gamma_r=\gamma$ for all $r\geq 0$ and we use Proposition \ref{prop:algo} to construct two graphs with the same asymptotic degree distribution $\boldsymbol{\tp}$, one with a positive clustering coefficient, the other with no clustering. We then compare the contagion thresholds in these two graphs.

\begin{figure}
\centering
\epsfig{file=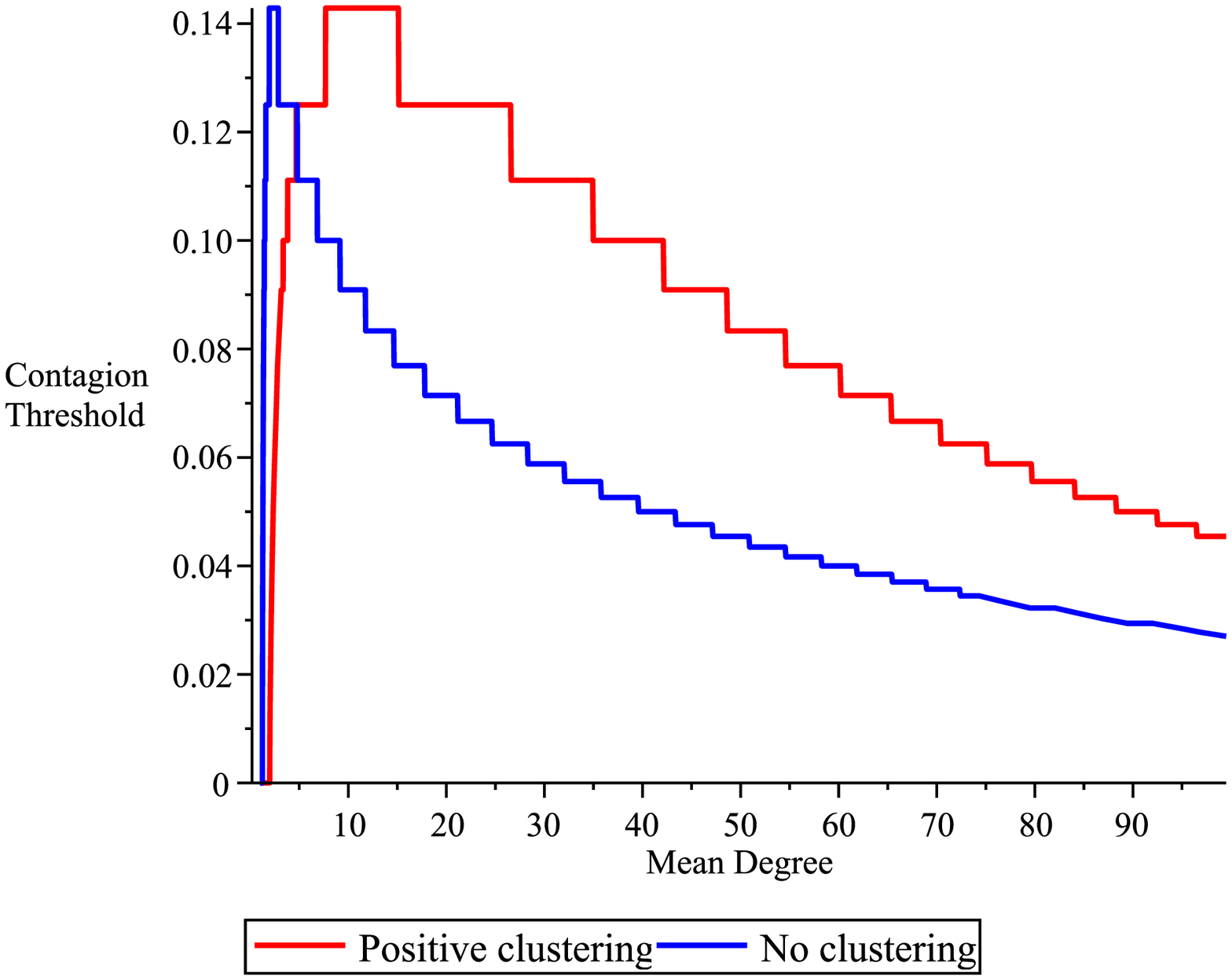, width=6cm}
\epsfig{file=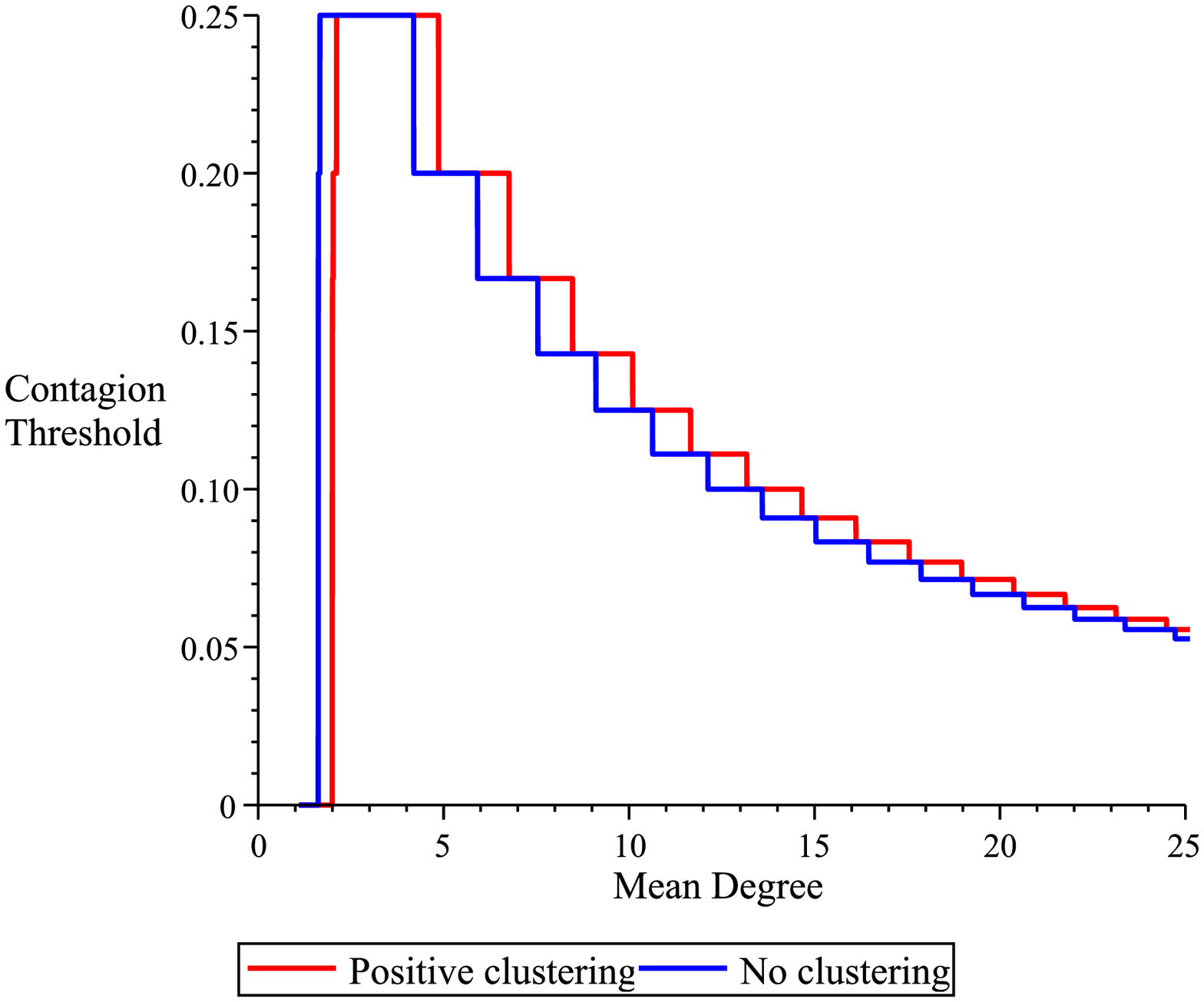, width=6cm}
\caption{On the left: Contagion thresholds in two graphs with the same degree distribution $\tilde{p}_r \propto r^{-\tau} e^{-r/50}$. On the right: Contagion thresholds in two graphs with the same degree distribution $\tilde{p}_r=e^{-\lambda} \lambda^{r-1}/(r-1)!$.}
\label{fig:ContPL}
\end{figure}

In Figure \ref{fig:ContPL} on the left, we consider a power law degree distribution with parameter $\tau>0$ and exponential cutoff: for all $r\geq 1$, $\tp_r \propto r^{-\tau} e^{-r/50}$. On the one hand, we consider (in red) the graph $G^{1}(\tau)$ for $C=C^{\max}$ (so that $\gamma=1$ and the distribution in the original graph is: $p_r\propto r^{-(\tau+1)} e^{-r/50}$). On the other hand, we consider (in blue) the graph $G^{0}(\tau)$ given by Proposition \ref{prop:algo} for $C=0$ (so that $\gamma=0$ and the distribution in the original graph is: $p_r=\tp_r$). In Figure \ref{fig:ContPL} (on the left), we make the parameter $\tau$ vary: the red (resp. blue) curve corresponds to the contagion threshold $q_c^{1}(\tau)$ (resp. $q_c^{0}(\tau)$) of the graph $G^{1}(\tau)$ (resp. $G^{0}(\tau)$), defined in \eqref{eq:qc}. Contagion thresholds are given with respect to the mean degree $\tl=\sum_r r\tp_r$ (that is a decreasing function of $\tau$).

In Figure \ref{fig:ContPL} on the right, we consider another form for the degree distribution $\boldsymbol{\tp}$: let $\lambda>0$, and set $\tilde{p}_r=e^{-\lambda} \lambda^{r-1}/(r-1)!$ for all $r\geq 1$. As before, we consider the graph $G^{1}(\lambda)$ for $C=C^{\max}$ ($\gamma=1$ and $\boldp$ is a Poisson distribution with parameter $\lambda$: $p_r=e^{-\lambda} \lambda^{r}/r!$), and the graph $G^{0}(\lambda)$ given by Proposition \ref{prop:algo} for $C=0$ ($\gamma=0$ and $p_r=\tp_r$). In Figure \ref{fig:ContPL} (on the right), we plot the contagion thresholds for these two graphs, with respect to the mean degree $\tl=\lambda+1$.

Both left and right-hand sides of Figure \ref{fig:ContPL} show that, when the mean degree $\tl$ of the graph is low, the contagion threshold $q_c^{0}$ of the graph with no clustering is greater than the threshold $q_c^{1}$ of the graph with positive clustering. Hence, if the parameter $q$ of the contagion process is in the interval $]q_c^{1},q_c^{0}[$, a global cascade is possible only in the graph with no clustering: in that case, the clustering 'inhibits' the contagion process. On the contrary, for high values of the mean degree, we have that $q_c^{0}<q_c^{1}$, so the clustering increases the range of parameter $q$ for which a global cascade is possible.

Let us fix the parameter $q\in(0,1)$ of the contagion ($q$ sufficiently low): this corresponds to a horizontal cut in Figure \ref{fig:ContPL}. The interval of mean degrees $\tl$ for which a global cascade is possible moves to the right when the clustering increases. Hence when the parameter of the contagion is fixed, clustering favors contagion processes on graphs with a higher mean degree.

Now we study more precisely what happens if we fix the mean degree in the graph (which corresponds to a vertical cut in Figure \ref{fig:ContPL}), and increase the clustering coefficient between $0$ and its maximal value $C^{\max}$ (even if we consider the first notion of clustering coefficient, the same phenomenon appears with the second notion).

\begin{figure}
\centering
\epsfig{file=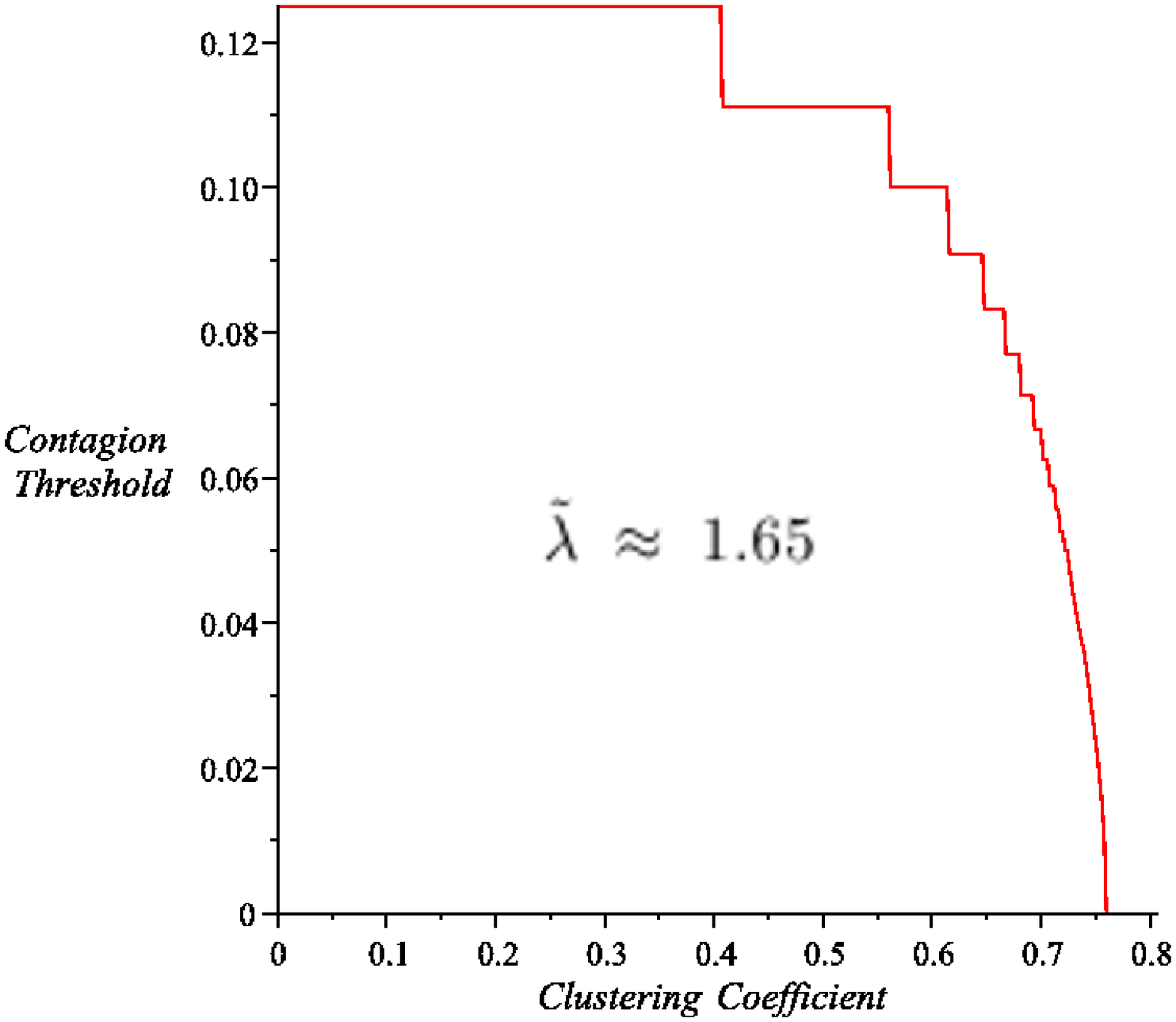, width=6cm}
\epsfig{file=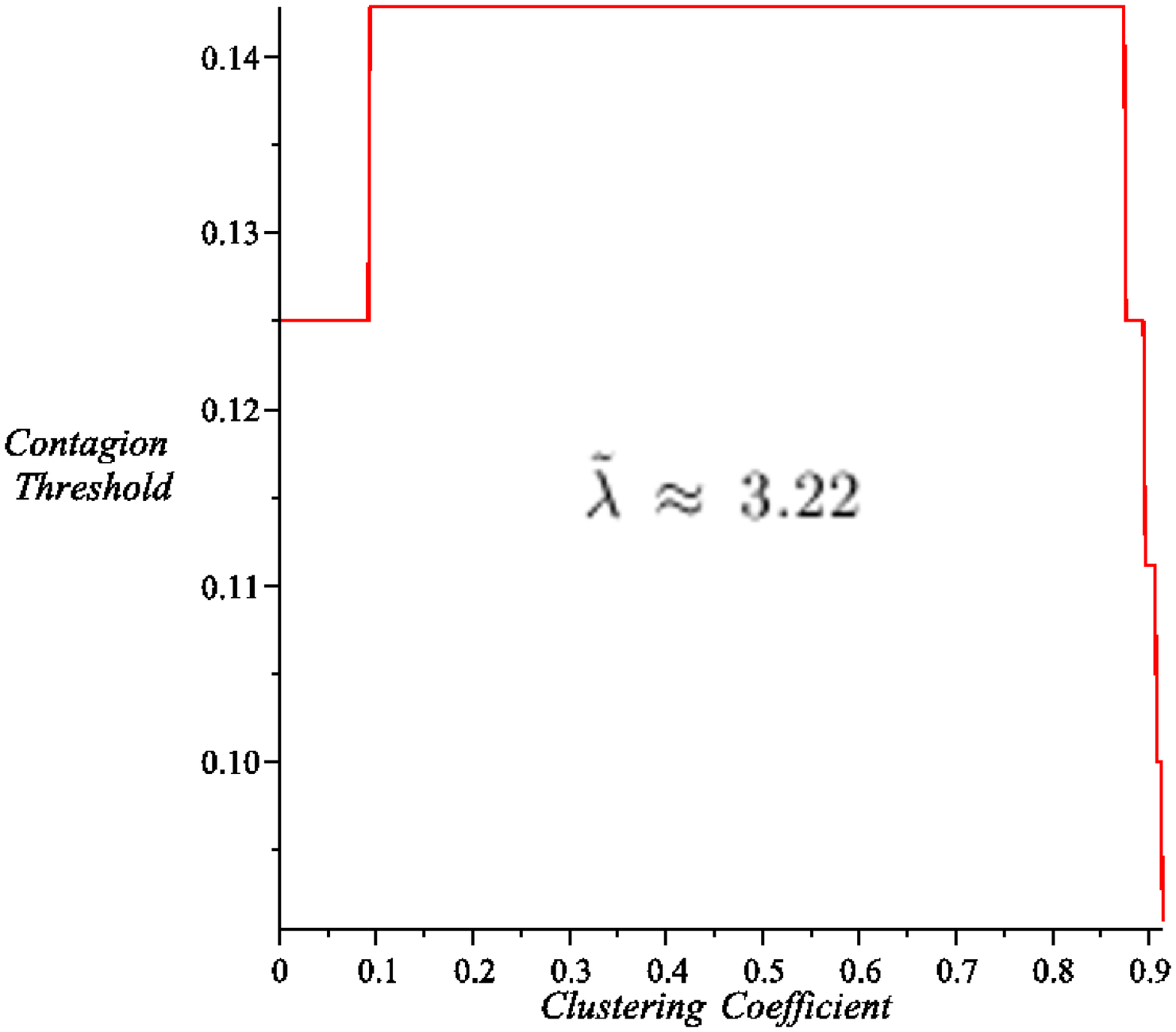, width=6cm}
\epsfig{file=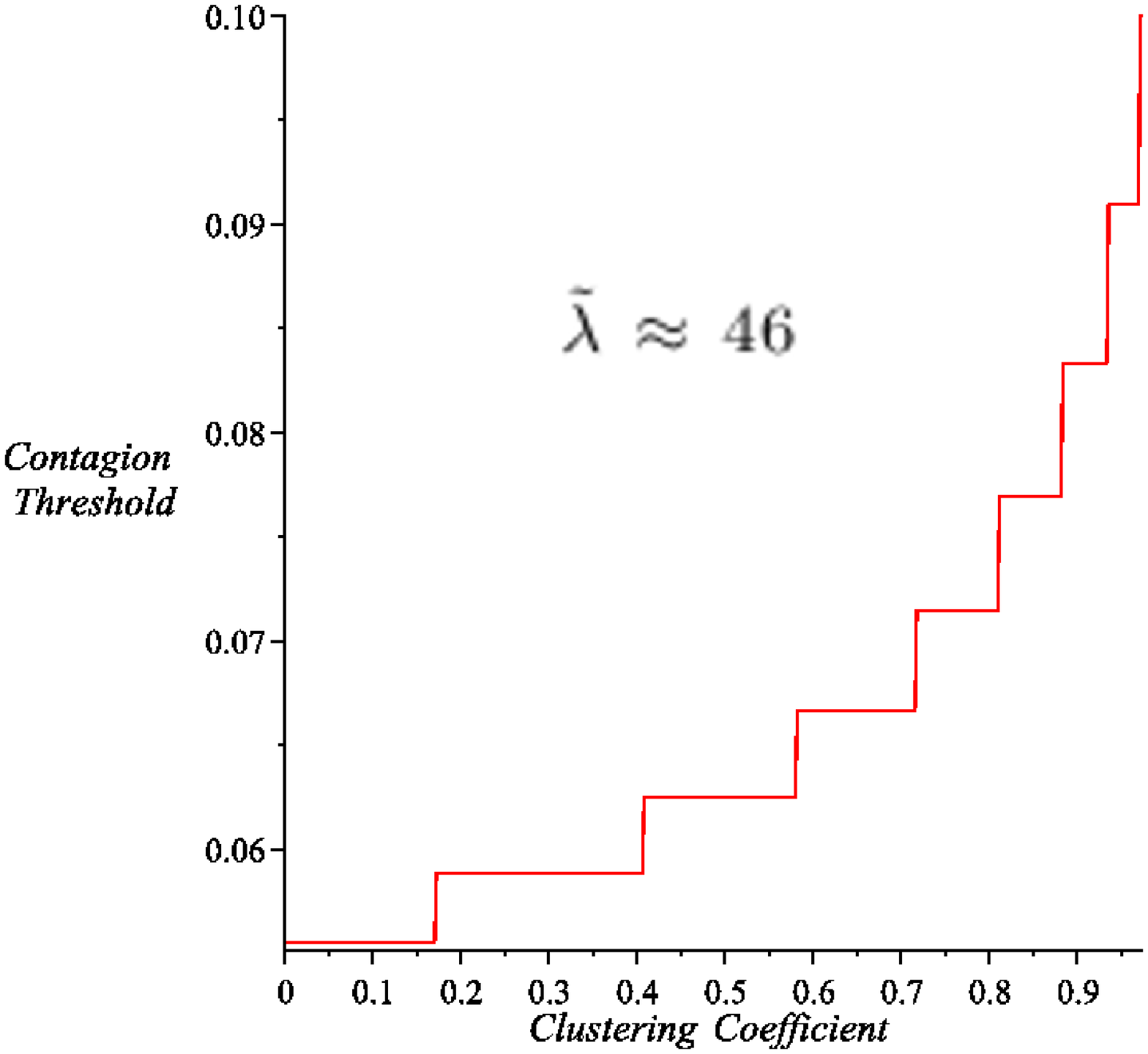, width=6cm}
\caption{Evolution of the contagion threshold in a graph with mean degree $\tl$ with respect to the clustering coefficient $C$ (for a fixed power law degree distribution).}
\label{fig:Cont0}
\end{figure}

In Figure \ref{fig:Cont0} on the top left corner (resp. top right corner, bottom), we consider a power law degree distribution with exponential cutoff: $\tp_r \propto r^{-\tau} e^{-r/50}$, with parameter $\tau=2.5$ (resp. $\tau=1.81$, $\tau=0.1$). We plot the contagion threshold $q_c$ for the graph given by Proposition \ref{prop:algo}, when the degree distribution is $\boldsymbol{\tp}$ and the clustering coefficient varies from $0$ to $C^{\max}$. We consider three different slices of Figure \ref{fig:ContPL} (left), and we go from the blue curve ($C=0$) to the red one ($C=C^{\max}$), progressively increasing the clustering coefficient. For a very low value of the mean degree ($\tl\approx 1.65$, top left corner of Figure \ref{fig:Cont0}), the contagion threshold decreases with the clustering. The opposite happens when the mean degree is very high ($\tl\approx 46$, bottom). In addition, for some intermediate values of the mean degree, as for $\tl\approx 3.22$ (top right corner), low values of the clustering 'helps' the contagion process, but, as the clustering coefficient becomes higher, the opposite happens: it 'inhibits' more and more the contagion process.

We see that the impact of clustering is different
for low values of the mean degree and for high values of the mean
degree. In the low values regime for the mean degree, the
contagion is more and more difficult, as the clustering increases. On the contrary, in the high values regime, the higher the clustering is, the more it 'helps' the contagion. When the value of the mean degree is exactly between these two cases, the effect of clustering is ambiguous: a low clustering coefficient 'helps' the contagion process, but a high one 'inhibits' the process.


\subsection{Effect of clustering on the cascade size for the contagion model}

We still consider the game-theoretic contagion model proposed by Morris \cite{mor} (described in the introduction), and the case where $\gamma_r=\gamma$ for all $r\geq 0$. In this subsection, the parameter $q\in(0,1)$ of the contagion process is fixed, and we want to highlight the effect of the clustering on the cascade size.

First we compare two graphs with the same asymptotic degree distribution $\boldsymbol{\tp}$, one having a positive clustering coefficient, the other having no clustering. In Figure \ref{fig:ContSize3}, we plot the sizes of the cascade and the pivotal players set for each of these graphs. 

\begin{figure}
\centering
\epsfig{file=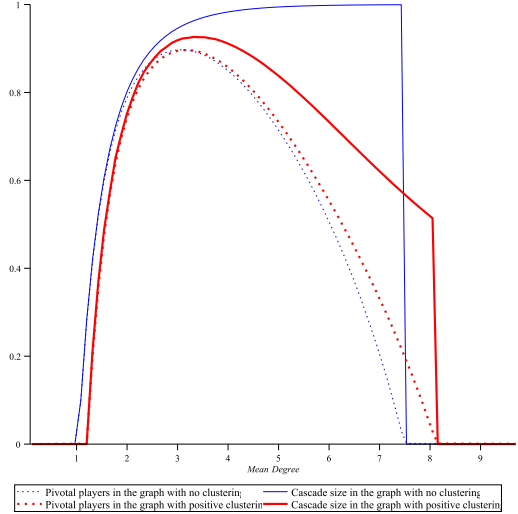, width=7cm}
\caption{Set of pivotal players and cascade sizes for $q=0.15$}
\label{fig:ContSize3}
\end{figure}

More precisely, in Figure \ref{fig:ContSize3}, we fix $q=0.15$. The red curves correspond to a graph with positive clustering, constructed as follows: we start from a Poisson distribution with parameter $\lambda$ for $\boldp$, and $\gamma=0.2$. This gives $\tilde{p}_r=\frac{0.2 r+0.8 }{0.2\lambda+0.8}\frac{e^{-\lambda}\lambda^r}{r!}$ and clustering coefficient $C=\frac{0.2\lambda}{0.2\lambda+1.2}>0$. The blue curves correspond to a graph with the same asymptotic distribution $\boldsymbol{\tp}$, but no clustering (in that case, $\boldp=\boldsymbol{\tp}$ and $\gamma=0$). We make the parameter $\lambda$ vary, and the sizes of the cascade (solid lines) and the pivotal players set (dot lines) are plot with respect to the mean degree $\tl=\sum_r r\tp_r$ in the graph. 

For each graph, we observe that there is a cascade if and only if the set of pivotal players is large, as explained in \ref{subs:contagion}. In addition, the interval of mean degrees $\tl$ for which a cascade is possible moves to the right when the clustering coefficient increases, which is consistent with our observations on Figure \ref{fig:ContPL}. Finally, we observe that the size of the cascade (when it exists) decreases with the clustering. This comes from the fact that cliques of degree $d\geq q^{-1}$ (\ie cliques with positive threshold) stop the contagion process (as explained in Lemma \ref{lem:cascade}). In the extremal case when $\gamma=1$ (each vertex of degree $d$ is replaced by a clique of size $d$), the cascade is exactly the set of pivotal players. When the probability $\gamma$ of replacing a vertex by a clique increases, the cascade triggered by a pivotal player becomes closer and closer to the set of pivotal players only (until it is exactly the set of pivotal players). This observation is confirmed in Figure~\ref{fig:ContSizeClust}.

\begin{figure}
\centering
\epsfig{file=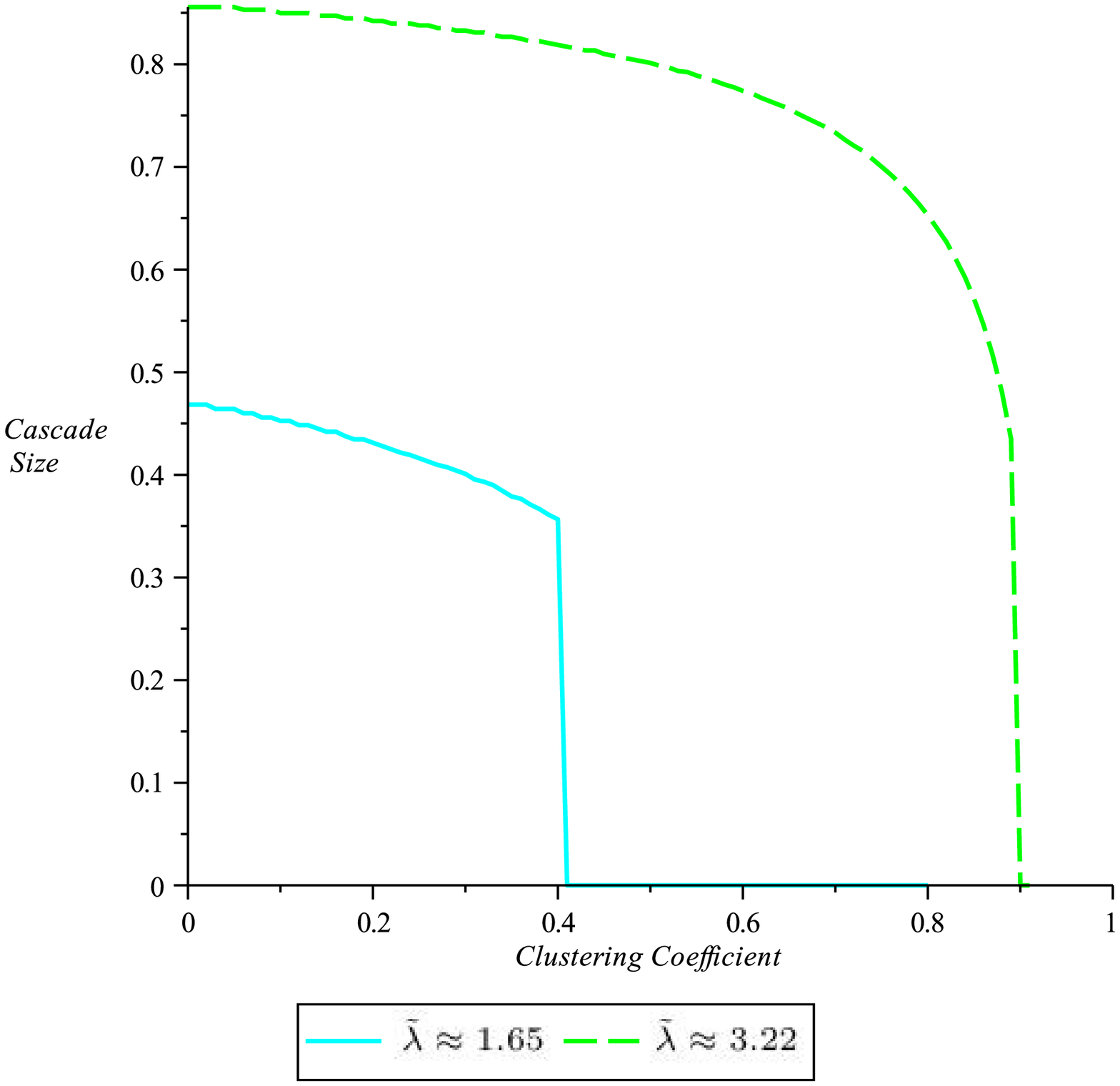, width=6cm}
\epsfig{file=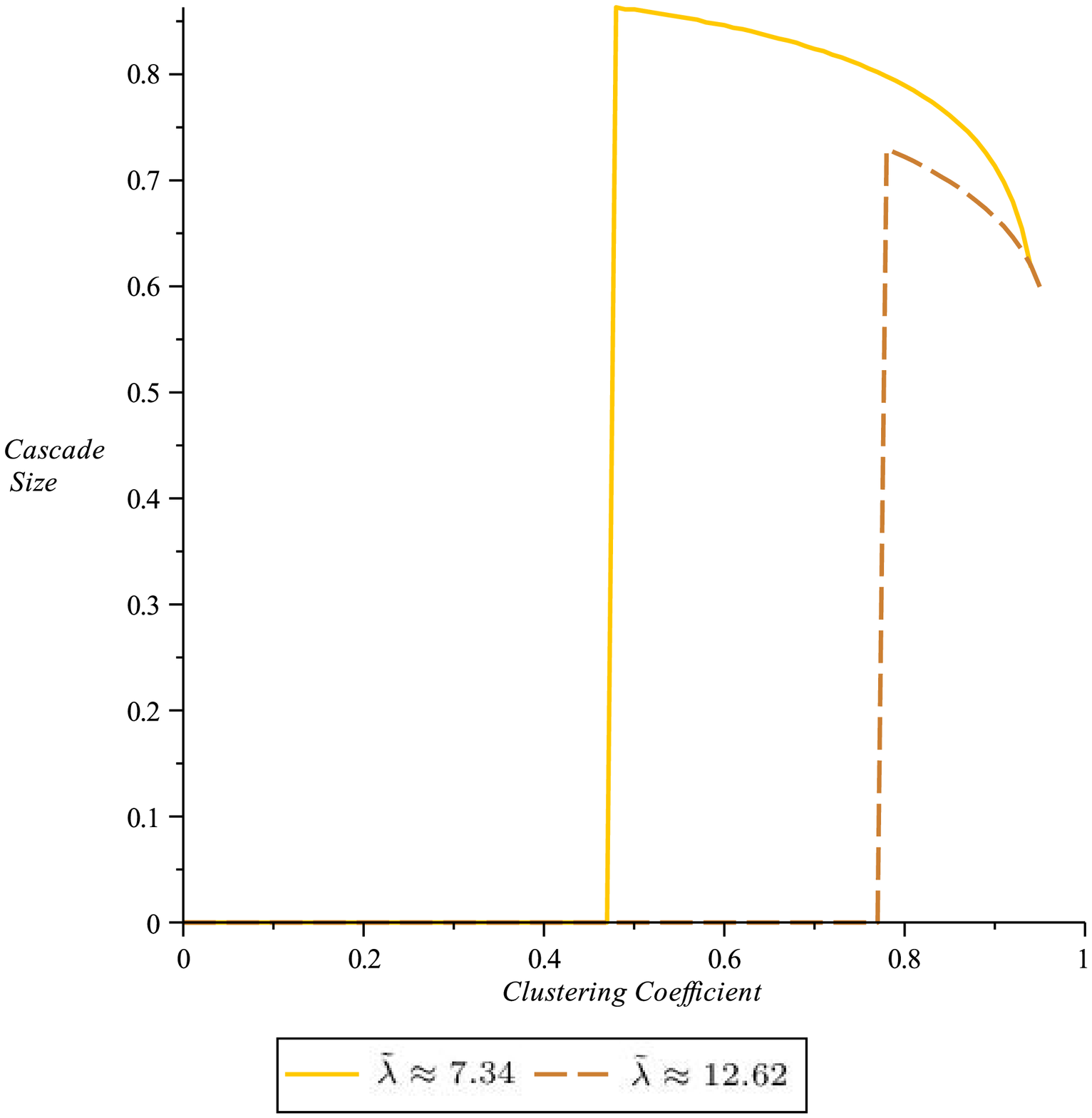, width=6cm}
\caption{Effect of the clustering on the cascade size for $q=0.12$, in a graph with a fixed power law degree distribution of mean $\tl$.}
\label{fig:ContSizeClust}
\end{figure}

To study more precisely the effect of clustering on the cascade size, we plot (in Figure~\ref{fig:ContSizeClust}) the cascade size for $q=0.12$, with respect to the clustering coefficient of a graph with power law degree distribution with exponential cutoff: $\tp_r \propto r^{-\tau} e^{-r/50}$, with parameter $\tau=2.5$ (resp. $\tau=1.81$, $\tau=1.3$, $\tau=1$), so that the mean degree of the graph is $\tl\approx 1.65$ (resp. $\tl\approx 3.22$, $\tl\approx7.34$, $\tl\approx 12.62$). Note that the values of the mean degree $\tl$ on the right-hand side of Figure \ref{fig:ContSizeClust} correspond to the case where the clustering 'helps' the contagion to spread, while the case $\tl\approx 1.65$ corresponds to the case where the clustering 'inhibits' the contagion, as detailed in the previous subsection. As for Figure \ref{fig:ContSize3}, we observe that the cascade size decreases with the clustering coefficient, when the cascade size is positive (\ie when a cascade is possible). The fact that a cascade is not possible for low values of clustering (right-hand side) comes from the fact that, for a fixed parameter $q$, the interval of $\tl$ for which a cascade is possible moves to the right, as observed in Figures \ref{fig:ContPL} and \ref{fig:ContSize3}.

\subsection{Phase transition for the symmetric threshold model with degree based activation}

In this subsection, we allow a positive fraction of nodes to be active at the beginning of the diffusion process. More precisely, on a given graph $G$, the set $S$ of initial active nodes is random, and each node of degree $d$ in $G$ belongs to $S$ with some probability $\alpha_d> 0$, independently for each node. We set $\balpha=(\alpha_d)_{d\geq 0}$. 

We define an adaptation of the usual degree based activation for the random graph $\tGr$ (so that the initial activation differs from the one in Subsection \ref{subs:diffA}). First we draw independent random variables for each vertex in the original graph $\Gr$. More precisely, for each vertex $i$ (of degree $d_i$ in $\Gr$), we draw a Bernoulli random variable $a(i)$ with parameter $\alpha_{d_i}$. When a vertex $i$ of $\Gr$ is replaced by a clique in $\tGr$, we associate to each vertex inside the clique the same activation variable $a(i)$ (if $i$ is not replaced by a clique, it keeps its own activation variable). Each vertex $v$ in $\tGr$ belongs to the initial seed $S$ if and only if $a(v)=1$.
Note that each node of degree $d$ in $\tGr$ belongs to $S$ with probability $\alpha_d> 0$ (since vertices inside the clique generated by $i$ have the same degree as $i$). Thus the only difference with the usual degree based activation is that activation variables are not independent inside a clique: either the whole clique belongs to the initial seed $S$, either no vertex in the clique belongs to $S$. 


Using the notation $\tgamma$ defined in Proposition \ref{prop:deg} and the binomial probabilities $b_{sr}(p)$ defined at the end of Section \ref{sec:intro}, we define (omitting the dependence on $\balpha$, $\boldt$, $\boldp$ and $\boldg$):
\begin{eqnarray}
L(z)&:=&\sum_{s} \frac{\left[s \gamma_s+(1-\gamma_s)\right] p_s}{\tgamma} \left[ (1-\alpha_s) t_{s0}(1-z^{s})+\alpha_s \right] \nonumber  \\ 
&& +\sum_{s} \frac{(1-\gamma_s) p_s}{\tgamma} (1-\alpha_s)\left(1-t_{s0}-\sum_{\ell\neq 0}t_{s\ell}\sum_{r\geq s-\ell}b_{sr}(z)\right), \nonumber  \\
 h(z)&:=&\sum_{s} (1-\alpha_s) sp_s \left[t_{s0}z^{s}+\gamma_s(1-t_{s0})z\right] \nonumber  \\ 
&& +\sum_{s} (1-\alpha_s) p_s (1-\gamma_s)\sum_{s\geq \ell\neq 0}t_{s\ell}\sum_{r\geq s-\ell} r b_{sr}(z),\nonumber \\
\label{eq:zetaContA} \zeta &:=&\sup\{\, z\in[0,1):\lambda z^2=h(z)\}.
\end{eqnarray}

\begin{theorem}
\label{th:contagionA}
Consider the random graph $\tGr$ for a sequence $\boldd$ satisfying
Condition \ref{cond} with probability distribution $\boldp=\pr$, and clustering parameter $\boldg=\cd$. Let $\boldt$ be a family of probability distributions, and $\boldk$ random thresholds drawn according to $\boldt$ in the original graph $\Gr$ (\ie if $i$ is a vertex in $\Gr$ replaced by a clique in $\tGr$, then all vertices in the clique have the same threshold $k(i)$).
We are given an activation set $S$ drawn according to the distribution $\balpha$ 
(so that vertices in the same clique are either all active or all inactive).
Then we have, for the symmetric threshold model defined in \ref{subs:def_cont}: if $\zeta=0$, or if $\zeta\in(0,1]$, and further $\zeta$ is such that there exists $\eps>0$ with $\lambda z^2<h(z)$ for $z\in (\zeta-\eps,\zeta)$, then we have 
that the size $C(\boldt,\balpha)$ of the active nodes at the end of the symmetric threshold process verifies:
\begin{eqnarray*}
 C(\boldt,\balpha)/\tn &\overset{p}{\longrightarrow}& L(\zeta).
\end{eqnarray*}
\end{theorem}

Heuristically, taking $\alpha_s=0$ for all $s$ in the definitions of the previous theorem allows to recover the result of Theorem \ref{th:contagion}. When $\gamma_r=0$ for all $r\geq 0$, we recover a result in \cite{lel:diff}.

If we apply this result to the case where thresholds are constant among nodes (\ie there exists an integer $k$ such that $k(v)=k$ for each vertex $v$), our model corresponds to a slight modification of the usual bootstrap percolation. Indeed the initial activation here is not independent among nodes that belong to the same clique. 

\section{Proofs} \label{sec:pf}

In the whole section, we consider a sequence $\boldd$ satisfying Condition \ref{cond} with probability distribution $\boldp=\pr$. 

\subsection{Configuration Model}

In order to prove Theorem \ref{th:diff}, it will be more convenient to work with the configuration model $\Grr$ (see for instance \cite{bb}): each vertex $i$, $1\leq i\leq n$, has $d_i$ half-edges, and the random graph $\Grr$ is obtained by taking a uniform matching among all possible matchings of half-edges into pairs. Conditioned on this multigraph being simple, it is distributed as $\Gr$. Condition \ref{cond} implies (in particular) that 
\begin{eqnarray} \label{eq:multi}
\liminf \Pb(\Grr \textrm{ is simple}) >0
\end{eqnarray}
(see \cite{ja:simple}), which allows to transfer directly results that hold in probability for $\Grr$ to the model $\Gr$.

As for the simple graph, we consider the model $\tGrr$: we associate to each $i\in\intn$ a Bernoulli variable $X(i)$ with parameter $\gamma_{d_i}$, all variables being independent. If $X(i)=1$, we replace node $i$ by a clique of size $d_i$ in which each vertex has exactly $d_i-1$ neighbors inside the clique, and one half-edge outside. Then we match half-edges as for $\Grr$. Hence $\tGrr$ is simple if and only if $\Grr$ is. So, conditioned on $\tGrr$ being simple, it is distributed as $\tGr$, and equation \eqref{eq:multi} implies that $$\liminf \Pb(\tGrr \textrm{ is simple}) >0.$$ Therefore, we can prove Theorems \ref{th:diff} and \ref{th:contagion} for one of the models $\tGrr$ or $\tGr$, and it will imply that they are true for both.

\subsection{Link between the graph $\tGr$ and the original graph $\Gr$}
 \label{subs:link}

Let $\tG$ be distributed as $\tGr$. We say that a vertex in $\tGr$ has parent $i\in\intn$ if it belongs to a clique that replaces the vertex $i$ of $\Gr$ (when $X(i)=1$) or if it is $i$ (when $X(i)=0$). For any subgraph $\tH\subset \tG$, we obtain the graph $\projj(\tH)$ by identifying in $\tH$ the vertices that have the same parent and that are connected in $\tH$. For instance, Figure \ref{fig:Proj} represents a clique of size $4$ in $\tG$ that comes from the replacement of a vertex $i$: thus all the vertices in the clique have the same parent $i$. In the subgraph $\tH$, some of the edges of the clique are not present (those in dots): the clique is split into two connected components. In the corresponding graph $\projj(\tH)$, we merge the vertices of the clique that are connected together.

\begin{figure}
\centering
\epsfig{file=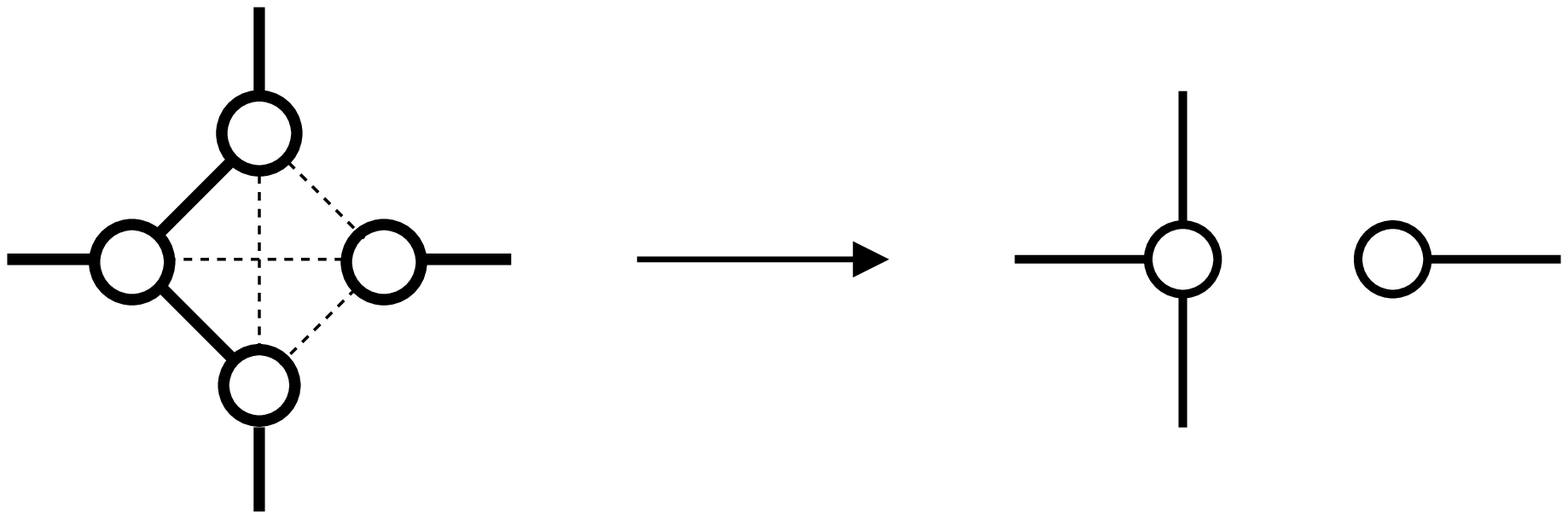, width=6cm}
\caption{Transformation of a subgraph in $\tGrr$ (or $\tGr$) by $\projj$}
\label{fig:Proj}
\end{figure}

We use the same definition of $\projj(\tH)$ when $\tH$ is a subgraph of a random graph distributed as $\tGrr$.

For any graph $G$, set $\nu(G)$ for the number of vertices in $G$.
The next lemma will be useful in several proofs.
\begin{Lemma}\label{lem:giant_comp}
Let $\tG$ be distributed as $\tGr$. Let $H$ be a subgraph of $\proj(\tG)$ such that $\nu(H)=o_p(n)$. Let $\tH$ be the maximal subgraph of $\tG$ such that $\proj(\tH)=H$. Then we have: $\nu(\tH)=o_p(\tn)$.
\end{Lemma}

\begin{proof}
 We can bound $\nu(\tH)/n $ using Cauchy-Schwarz inequality:
\begin{eqnarray*}
 \nu(\tH)/n &\leq& \sum_{r} r \nu_r(H)/n \\
&\leq& \sqrt{\sum_{r} \frac{r^2 \nu_r(H)}{n} } \sqrt{\sum_{r} \frac{\nu_r(H)}{n} } \\
&\leq& \sqrt{ \frac{\sum_{i} d_i^2}{n} } \sqrt{ \frac{\nu(H)}{n} } \\
\end{eqnarray*}

Yet Condition \ref{cond}-\textit{(iii)} implies that $ \sum_{i} d_i^2/n=O(1)$, and by hypothesis $\nu(H)/n \to_p 0$, so $ \nu(\tH)/n \to_p 0$, and  $ \nu(\tH)/\tn \to_p 0$ due to Proposition \ref{prop:deg}.
\end{proof}

\subsection{Proof of Theorem \ref{th:diff}}
\label{subs:steps_diff}

A heuristic (using a branching process approximation) is given in Appendix \ref{app:BP}. We first give the idea of the proof.

Let $\tG$ be distributed as $\tGrr$ and $\pi\in[0,1]$. In the percolated graph $\tG_{\pi}$, the removal of some edges inside a clique can split the clique into several connected components.
In order to study the percolated graph $\tG_{\pi}$ (as described in Section \ref{sec:diff}), we proceed in three steps: 
\begin{itemize}
 \item[\textit{Step 1.}] We consider only the edges that are \textit{inside} a clique, and we delete independently each of them with probability $1-\pi$. The graph we obtain is called $\tG_{\pi}^{(1)}$. We are interested in the graph $G'=\projj(\tG_{\pi}^{(1)})$. If we condition $G'$ on its number of vertices $n'$ and its degree sequence $\boldd'$, we have that $G'$ is distributed as $G^*\left(n',\boldd'\right)$. The first step consists in computing the asymptotic distribution of the degree sequence $\boldd'$.
\item[\textit{Step 2.}] Then we delete independently with probability $1-\pi$ each edge of $G'$, and we apply results of \cite{ja09} in order to study the components' sizes in the percolated graph $G'_{\pi}$.
\item[\textit{Step 3.}] We deduce the components' sizes in $\tG_{\pi}$ from the previous step, using the fact that $\projj(\tG_{\pi})$ is distributed as $G'_{\pi}$.
\end{itemize}

In the following, when we consider the model $\tGrr$, we take the multiplicity of edges into account when we compute the degree of a vertex. More precisely, we say that a vertex in $\Grr$ or $\tGrr$ has 'degree' $d$ if it has $d$ (simple) half-edges. For instance, each loop of a given vertex has contribution $2$ in its degree.


\textit{Step 1.} For $d\geq 1$, let $V^{(n)}_{d}$ be the set of vertices $i$ in $\Gr$ with degree $d$ and such that $X(i)=1$: $i$ is replaced by a clique $K(i)$ of size $d$ in $\tG$. Let $K(i,\pi)$ be the subgraph of $K(i)$ obtained after a bond percolation with parameter $\pi$. We consider the subgraph $\tF_d(\pi)\subset\tG_{\pi}$ that contains the percolated version of the cliques with initial size $d$:
\begin{eqnarray*}
 \tF_d(\pi)=\bigcup_{i\in V^{(n)}_{d}} K(i,\pi)
\end{eqnarray*}
The next lemma gives the limit, as $n\to\infty$, for the number $N^{(n)}(d,k,\pi)$ of connected components in $\tF_d(\pi)$ whose size is $k\leq d$.

\begin{Lemma}\label{lem:diff_comp}
For any $d\geq 1$ and $k\leq d$, we have that: 
\begin{eqnarray*}
N^{(n)}(d,k,\pi)/n \overset{p}{\longrightarrow} \frac{d}{k} f(d,k,\pi) p_d \gamma_d ,
\end{eqnarray*}
where $f(d,k,\pi)$ is given by \eqref{eq:deff}. 
\end{Lemma}

\begin{proof}
For each vertex $i$ in $V^{(n)}_{d}$, we label the vertices of $K(i,\pi)$ from $1$ to $d$. We look at all the vertices with label $1$, and we let $M^{(n)}(d,k,\pi)$ be the number of such vertices whose connected component in $K(i,\pi)$ has size $k$. Using the Law of Large Numbers and the fact that $|V^{(n)}_{d}|/n\to_p p_d \gamma_d$, we have that $M^{(n)}(d,k,\pi)/n\to_p f(d,k,\pi) p_d \gamma_d $, where $f(d,k,\pi)$ is by definition the probability that the component of 1 contains $k$ vertices. So the total number of vertices in $\tF_d(\pi)$ that belongs to a component of size $k$ is: $d f(d,k,\pi) p_d \gamma_d n +o_p(n)$ and, in order to have the number of such components, we have to divide by $k$, which proves the lemma.
\end{proof}

Let $\tG_{\pi}^{(1)}$ be the graph obtained from $\tGrr$ when we replace each vertex $i$ such that $X(i)=1$ by the percolated clique $K(i,\pi)$.
For any $k\geq 0$, let $n'_k$ be the number of vertices with 'degree' $k$ in the projected graph $G'=\projj(\tG_{\pi}^{(1)})$.
In order to compute $n'_k$, we have to consider the vertices $i$ such that $X(i)=0$ (there are $n_k-|V^{(n)}_{k}|$ such ones, where $n_k$ is the number of vertices with 'degree' $k$ in $\Grr$), and the vertices that come from a clique of initial size $d$, for some $d\geq k$ (each such vertex corresponds to a component of size $k$ in $\tF_d(\pi)$, so there are $N^{(n)}(d,k,\pi)$ such ones). This gives the following relation, for all $k\geq 0$:
\begin{eqnarray*}
 n'_k=n_k-|V^{(n)}_{k}|+\sum_{d\geq k} N^{(n)}(d,k,\pi),
\end{eqnarray*}
So Lemma \ref{lem:diff_comp} gives the following asymptotic distribution for the degree sequence $\boldd'$:
\begin{Lemma}\label{lem:diff_asymptotic}
Let $n':=\sum_k n'_k$ be the total number of vertices in $G'$. Then the proportion of vertices with degree $k$ in $G'$ has the following limit, as $n\to\infty$: 
\begin{eqnarray*}
\frac{n'_k}{n'} &\overset{p}{\longrightarrow}& p'_k:=\frac{\varrho_k}{\sum_{\ell}\varrho_{\ell}}
\end{eqnarray*}
where $\varrho_k:=p_k(1-\gamma_k)+\sum_{d\geq k} \frac{d}{k} f(d,k,\pi) p_d \gamma_d $
\end{Lemma}

In addition, the uniform summability of $kn_k/n$ implies the uniform summability of $kn'_k/n'$, so that $\sum_k kn'_k/n' \to_p \sum_k k p'_k$.

\textit{Step 2.} We apply Theorem 3.9 in \cite{ja09} to the random graph $G'$. Indeed, we can assume without loss of generality that the previous convergences (Lemma \ref{lem:diff_asymptotic}) hold a.s., and not just in probability (as in \cite{ja09}: using the Skorohod coupling theorem, see \cite{kal} for instance, or arguing by selecting suitable subsequences). Then there is a giant component in the percolated graph $G'_{\pi}$ if and only if 
$$ \pi\sum_d d(d-1)p'_d>\sum_d d p'_d, $$ which is equivalent to the fact that $\pi \Eb\left[ \cK(D^*+1,\pi,\boldg)-1\right]>1$.

\textit{Step 3.} The proof of \textit{(ii)} follows easily from the previous step, and Lemma \ref{lem:giant_comp}. 

We give the main lines of the proof of \textit{(i)}. Assume $\pi>\pi_c$, which corresponds to $\pi\sum_d d(d-1)p'_d>\sum_d d p'_d$. Let $\cC_1$ be the largest connected component in $G'_{\pi}=\proj(\tG_{\pi})$ and $\bcC_1$ be the connected component of $\tG_{\pi}$ such that $\proj(\bcC_1)=\cC_1$.

We first compute the limit of $\nu(\bcC_1)/\tn$ as $n\to\infty$. 
Let $g$ be the generating function
\begin{eqnarray*}
\label{def:g}
 g(x) &:=& \sum_k p'_k x^k = \frac{1}{\varrho} \sum_k \varrho_k x^k,
\end{eqnarray*} 
and recall that its mean is called $\mu = \sum_k k \varrho_k / \varrho$.
Results in \cite{ja09} show that the number $\nu_r(\cC_1)$ of vertices with degree $r$ in $\cC_1$ satisfies: $\nu_r(\cC_1)/n' \to_p \sum_{\ell\geq r}b_{\ell r}(\sqrt{\pi})p'_{\ell}(1-\xi^r)$, where $n'$ is the total number of vertices in the graph $G'$, and $\xi$ is the unique $\xi\in(0,1)$ such that
\begin{eqnarray} \label{eq:xi_diff}
 g'(1-\pi^{1/2}+\pi^{1/2}\xi)=\mu(1-\pi^{-1/2}+\pi^{-1/2} \xi).
\end{eqnarray}
Since $h(z)=(1-\pi+\pi z)\cdot g' (1-\pi+\pi z)$, we have that $\xi$ is the solution of \eqref{eq:xi_diff} if and only if $\zeta=1-\pi^{-1/2}+\pi^{-1/2}\xi$ is the solution of $\mu \zeta(1-\pi+\pi \zeta)=h(\zeta)$. Note that we used the second notation in the statement of Theorem \ref{th:diff}, in order to be consistent with the notations of Theorem \ref{th:diffA} (in fact, we could have used either results in \cite{ja09} or Theorem 11 of \cite{lel:diff} for the current proof).

Unfortunately, we cannot deduce directly the size of $\bcC_1$ using only the asymptotic for $\nu_r(\cC_1)/n' $, $r\geq 0$. We have to be more precise: let $\nu_r^0(\cC_1)$ (resp. $\nu_r^1(\cC_1)$) be the number of vertices $i$ with degree $r$ in $\cC_1$ such that $X(i)=0$ (resp. $X(i)=1$). Then we have:
\begin{eqnarray*}
\nu_r^0(\cC_1)/n' &\to_p& \sum_{\ell\geq r}b_{\ell r}(\sqrt{\pi})p_{\ell}(1-\gamma_{\ell}) (1-\xi^r)/\varrho, \\
\nu_r^1(\cC_1)/n' &\to_p& \sum_{\ell\geq r}b_{\ell r}(\sqrt{\pi}) \sum_{d\geq \ell} \frac{d}{\ell} f(d,\ell,\pi) p_{d}\gamma_{d} (1-\xi^r)/\varrho.
\end{eqnarray*}
In summations, $d$ represents the degree of vertices in the initial graph $\tG$, $\ell$ the degree of vertices in $\tG^{(1)}_{\pi}$ (after the percolation \textit{inside} cliques), and $r$ the degree of vertices after the percolation on \textit{external} edges. In order to recover $\nu(\bcC_1)$, we have to multiply each term in $\nu_r^1(\cC_1)$ by $\ell$, and then sum over all $r$, which gives (exchanging summations on $r$ and $\ell$):
\begin{eqnarray*}
 \nu(\bcC_1)/n' &\overset{p}{\longrightarrow}& \frac{1}{\varrho} \sum_{k\geq 1} \sigma_k \left(1-(1-\pi^{1/2}+\pi^{1/2}\xi)^k \right). 
\end{eqnarray*}
Using that $n'/n\to_p \varrho$ and $\tn/n\to_p \tgamma$, we obtain that $\nu(\bcC_1)/\tn\to_p L(\zeta)$.

Let $\tC_1$ be the largest component in $\tG_{\pi}$. Adding cliques changes the sizes of the connected components: 
hence we have to prove that $\bcC_1=\tC_1$ whp. Let $\tC$ be any other component of $\tH$ different from $\bcC_1$. Its projection $\cC=\proj(\tC)$ is different from $\cC_1$, so Theorem 3.9 in \cite{ja09} implies that $\nu(\cC)/n \to_p 0$. Using Lemma \ref{lem:giant_comp} with $H=\cC$ shows that $\nu(\tC)/\tn \to_p 0$. Hence $\bcC_1$ is the largest connected component of $\tG$ whp, which ends the proof.


\subsection{Proof of Theorem \ref{th:diffA}}

The difference with the previous proof is the following: instead of using Theorem 3.9 of \cite{ja09} in steps 2 and 3, we use Theorem 10 of \cite{lel:diff}.

Indeed the first step is the same: the graph $G'=\projj(\tG_{\pi}^{(1)})$ (where $\tG_{\pi}^{(1)}$ is the graph obtained from $\tGr$ after a bond percolation on the edges \textit{inside} cliques only) has asymptotic degree distribution $\boldp'=(p'_k)_k$, with $p'_k=\varrho_k/\varrho$.

We apply (a slight extension of) Theorem 10 in \cite{lel:diff} for the graph $G'$ (with $t_{s\ell}=\Indb_{\{\ell=0\}}$). Let $\nu_s^0$ be the number of vertices $i$ such that $X(i)=0$ and that satisfy: the degree of $i$ in $G'$ (that is to say before the bond percolation in $G'$) is $s$, and $i$ is active at the end of the process. Let $\nu_{ds}^1$ be the number of vertices $i$ such that $X(i)=1$ and that satisfy: the degree of $i$ in the original graph $\proj(\tG)$ is $d$, the degree of $i$ in $G'$ is $s\leq d$, and $i$ is active at the end of the process. The probability that such a node $i$ (with degree $d$ in $\proj(\tG)$ and $s$ in $G'$) does not belong to the original seed $S$ is $(1-\alpha_d)^s$ (and initial activations are independent among nodes). Hence we have:
\begin{eqnarray*}
\nu_s^0/n' &\to_p& p_s(1-\gamma_s)\left[1-(1-\alpha_s)(1-\pi+\pi \zeta)^s\right]/\varrho, \\
\nu_{ds}^1/n' &\to_p&  \frac{d}{s} f(d,s,\pi) \gamma_d p_d \left[1-(1-\alpha_d)^s(1-\pi+\pi \zeta)^s\right]/\varrho,
\end{eqnarray*}
where $\zeta$ is given by \eqref{eq:zetaDiffA}. In order to obtain $C^{b}(\pi,\balpha)$, we have to multiply $\nu_{ds}^1$ by $s$, and sum over all $d$ and $s$, which gives: 
\begin{eqnarray*}
 C^{b}(\pi,\balpha)/n' &\to_p& \sum_s p_s(1-\gamma_s)\left[1-(1-\alpha_s)(1-\pi+\pi \zeta)^s\right]/\varrho \\
&& + \sum_{d\geq s} d f(d,s,\pi) \gamma_d p_d \left[1-(1-\alpha_d)^s(1-\pi+\pi \zeta)^s\right]/\varrho
\end{eqnarray*}
and ends the proof (since $n'/n\to_p \varrho$ and $\tn/n\to_p \tgamma$).

\subsection{Proof of Theorem \ref{th:contagion}}
\label{subs:pf_cont}

As before, we say that a vertex in $\tGr$ has parent $i\in\intn$ if it belongs to a clique that replaces the vertex $i$ of $\Gr$ (when $X(i)=1$) or if it is $i$ (when $X(i)=0$). 
For any graph $G$ and any vertex $v$ of $G$, let $D(v,\boldt)$ be the subgraph of $G$ induced by the final set of active vertices, when $v$ is the only vertex in the initial seed. With the notations of \ref{subs:contagion}, the number of vertices in $D(v,\boldt)$ is $C(v,\boldt)$. When $H$ is a subgraph of $G$, we set $D(H,\boldt)$ for the subgraph induced by the final set of active vertices in $G$, when the initial vertices in the seed are those of $H$.

We can make a comparison between an epidemic starting from a vertex $u$ in $\tGr$, and the epidemic that would have been generated by its parent $i$ in $\Gr$ (recall that thresholds are drawn such that a vertex $u$ has the same threshold as its parent $i$).
\begin{Proposition}\label{prop:cascade}
Let $u$ be a vertex of $\tG=\tGr$. Let $i$ be its parent, and $K$ be the clique generated by $i$ if $X(i)=1$ (otherwise, set $K=\{u\}$). Then we can bound the epidemic generated by $u$ the following way:
\begin{eqnarray*}
 \proj\left(D(u,\boldt)\right)\subset\proj\left(D(K,\boldt)\right)\subset D(i,\boldt).
\end{eqnarray*}
\end{Proposition}

The proof follows from Lemma \ref{lem:cascade} in \ref{subs:contagion}. In addition, we have the following lemma, that is a consequence of \cite{lel:diff}:
\begin{Lemma}\label{lem:leldiff}
Assume $\sum_{r} r(r-1) p_r t_{r0}< \sum_r  rp_r$. Let $u$ be a vertex chosen uniformly at random among the vertices of $\tG=\tGr$, and let $i$ be the parent of $u$. Then the size of the epidemic generated by $i$ in $\proj(\tG)$ is $C(i,\boldt)=o_p(n)$.
\end{Lemma}

\begin{proof}
 We cannot use directly the result of \cite{lel:diff} that says that, if $\sum_{r} r(r-1) p_r t_{r0}< \sum_r  rp_r$ and $i$ is chosen uniformly at random in $\Gr$, then $C(i,\boldt)=o_p(n)$. The idea is to apply Theorem 10 \cite{lel:diff}, with a parameter $\balpha=(\alpha_d)_{d=0}^{\infty}$ that satisfies: $\alpha_d=(d \gamma_d+1-\gamma_d)\alpha$ for all $d$, $\alpha$ being a positive constant. Then the same arguments as for the proof of Theorem 11-\textit{(ii)} \cite{lel:diff} work.
\end{proof}

This allows to deduce easily the case \textit{(ii)} of Theorem \ref{th:contagion}: assume $\sum_{r} r(r-1) p_r t_{r0}< \sum_r  rp_r$. 
Then combining Proposition \ref{prop:cascade} and Lemma \ref{lem:leldiff} gives that the number of vertices in $\proj\left(D(u,\boldt)\right)$ is $o_p(n)$ if $u$ is chosen uniformly at random in $\tG$. Applying Lemma \ref{lem:giant_comp} with $H=D(u,\boldt)$ concludes the proof of case \textit{(ii)}.

We now assume that the cascade condition \eqref{eq:casc} is satisfied. 


The proof of \eqref{eq:pivotal_players} is a consequence of a result from \cite{lel:diff}. Indeed let $\tG$ be distributed as $\tGr$. Let $\tH$ (resp. $H$) be the subgraph of $\tG$ (resp. $\proj(\tG)$) induced by the vertices of threshold zero. Note that $\proj(\tH)=H$.
We use Theorem 11 in \cite{lel:diff} for the graph $\proj(\tG)$ (with parameter $\pi=1$): it gives the components' sizes in $H$. Let $\cC_1$ (resp. $\tC_1$) be the largest connected component in $H$ (resp. $\tH$). The number $\nu_r(\cC_1)$ of vertices with degree $r$ in $\cC_1$ is computed in the proof of Theorem 11 in \cite{lel:diff}: $\nu_r(\cC_1)/n \to_p p_r t_{r0}(1-\xi^r)$, where $\xi$ is defined in \eqref{eq:xi_cont}. Hence we can deduce the size of the connected component $\bcC_1$ in $\tH$ such that $\proj(\bcC_1)=\cC_1$: $\nu(\bcC_1)/\tn \to_p \sum_{d}  \left[d \gamma_d+(1-\gamma_d)\right] p_d t_{d0}(1-\xi^{d})/\tgamma$. The way to show that $\bcC_1=\tC_1$ whp is similar to the end of Theorem \ref{th:diff}. This ends the proof of \eqref{eq:pivotal_players}.

The idea for the rest of the proof is to make a coupling between the epidemic on $\tG$ (with threshold distribution $\boldt$), and an epidemic on $\proj(\tG)$ with a different threshold distribution, that we call $\boldt'={(t'_{s\ell})}_{s,\ell}$. 

\begin{Proposition} \label{prop:link_sym_thr_model}
 Assume the epidemic on $\tG=\tGr$ starts from a vertex $u$ that has threshold zero, and let $i$ be the parent of $u$ in $\proj(\tG)$. We consider the following distribution of thresholds ${(t'_{s\ell})}_{0\leq \ell\leq s}$ for each $s\geq 0$:
\begin{itemize}
 \item $t'_{s0}=t_{s0}$;
\item  $t'_{s\ell}=(1-\gamma_s)t_{s\ell}$ for all $0<\ell<s$;
\item  $t'_{ss}=(1-\gamma_s)t_{ss}+\gamma_s(1-t_{s0})$.
\end{itemize}
Then there exist random thresholds $(k'(j))_{1\leq j\leq n}$ with the distribution $\boldt'=(t'_{s\ell})_{0\leq \ell\leq s}$ defined above
 such that
\begin{eqnarray*}
 \proj(D(u,\boldt))=D(i,\boldt'),
\end{eqnarray*}
where $D(i,\boldt')$ is the subgraph induced by the final set of active vertices in the symmetric threshold model starting from $i$ in $\proj(\tG)$, with threshold distribution $\boldt'={(t'_{s\ell})}_{s,\ell}$.
\end{Proposition}

\begin{proof}

Note that each vertex $i$ of $\proj(\tG)$ has two thresholds: $k(i)$ (that we used to define the epidemic on $\tG$), and the new threshold $k'(i)$, that we will define to use to make a comparison with the epidemic on $\tG$. Until the end of the proof, if no precision is given, we refer to $k'(i)$ as the threshold of $i$.

We explicit the natural coupling between the symmetric threshold model with parameter $\boldt$ in $\tG$ and the one with parameter $\boldt'$ in $\proj(\tG)$. If $u$ belongs to a clique, then the whole clique becomes active at the next step, so we can start the epidemic in $\proj(\tG)=\Gr$ from the parent $i$ of $u$. Let $v$ be the neighbor of $u$ outside the clique, and let $j$ be the parent of $v$. If the threshold $k(v)$ of vertex $v$ in $\tG$ is zero, then $v$ (and its whole clique if it has one) becomes active (see Lemma \ref{lem:cascade}). In this case, we choose $k'(j):=0$ for the threshold of $j$ in $\proj(\tG)$, so that it becomes also active in $\proj(\tG)$. If $k(v)>0$, then there are two cases: 
\begin{itemize}
 \item If $X(j)=1$, vertex $v$ and its clique stay inactive (Lemma \ref{lem:cascade}). In this case, we choose $k'(j):=s$ for the threshold of $j$ (so that it stays inactive).
\item If $X(j)=0$, vertex $v$ becomes active if and only if it has at least $k(v)+1$ active neighbors. So we set $k'(j):=k(v)=k(j)$ for the threshold of $j$.
\end{itemize}
Since the random variables $X(j)$, for $j$ in $\proj(\tG)$, are independent, the thresholds we associate to each node are also independent. In addition, we can easily verify that the conditional probability distribution of thresholds (knowing that the degree of the node is $s$) is given by ${(t'_{s\ell})}_{0\leq \ell\leq s}$. In fact, the epidemic we consider in $\proj(\tG)$ is almost the same as the one with parameter $\boldt$, except that we randomly put some nodes $j$ (those such that $X(j)=1$) to a threshold so high that they stay inactive.
\end{proof}

More precisely, let $C_{s\ell}(u,\boldt)$ (resp. $C'_{s\ell}(i,\boldt')$) be the final number of active vertices with degree $s\geq 0$ and threshold $\ell$ at the end of the symmetric threshold epidemic on $\tG$ (resp. $\proj(\tG)$), with threshold parameter $\boldt$ (resp. $\boldt'$), when the only vertex in the initial seed is $u$ (resp. $i$). Then, using the coupling described above, we have the following result, for each degree $s\geq 0$:
\begin{itemize}
 \item  $C_{s0}(u,\boldt)=C'_{s0}(i,\boldt')[sY_s+(1-Y_s)]$, where $Y_s$ is the proportion of vertices $j$ in $\proj(\tG)$ such that $X(j)=1$, among those that have degree $s$ and that belong to the cascade triggered by $i$.
\item For all $\ell\neq 0$, we have that $C_{s\ell}(u,\boldt)=C'_{s\ell}(i,\boldt')$, since the vertices of positive threshold that belong to the cascade triggered by $u$ are exactly those that are not replaced by a clique.
\end{itemize}

We have that $Y_s/n\to_p\gamma_s$ for all $s$, and the limit for $C'_{s\ell}(i,\boldt')$ is given by the following lemma (which is a slight extension of Theorem 11 \cite{lel:diff}):


\begin{Lemma} \label{lem:sym_thr_model}
 Assume (using the notations of Theorem \ref{th:contagion}) that $\zeta=0$ or $\zeta$ is such that there exists $\eps>0$ with $\lambda z^2<h(z)$ for $z\in (\zeta-\eps,\zeta)$. Then, for any $i$ that belongs to the set of pivotal players in $\proj(\tG)$, we have:
%
\begin{eqnarray*}
C'_{s\ell}(i,\boldt')/n & \to_p & p_s t'_{s\ell} \left(1-\sum_{r\geq s-\ell}b_{sr}(\zeta)\right).
\end{eqnarray*}
In particular, for $\ell=0$, we have: $C'_{s0}(i,\boldt')/n \to_p  p_s t'_{s0} \left(1-\zeta^s\right)$.
\end{Lemma}

\begin{proof}
 By slight extension of Theorem 11 \cite{lel:diff}, the number of inactive nodes with original
degree $s$, degree $r$ in the graph of inactive nodes and threshold
$\ell$ tends to $\sum_{i\geq s-r-\ell}p_s t'_{s\ell}b_{sr}(\zeta)b_{s-r,i}(0)=p_s t'_{s\ell}b_{sr}(\zeta) \Indb \{r\geq s-\ell\}$.
Hence summing over $r$ gives that the number of inactive nodes with original degree $s$ and threshold $\ell$ tends to $p_s t'_{s\ell} \sum_{r\geq s-\ell}b_{sr}(\zeta)$, which ends the proof.
\end{proof}

We assume that $\zeta=0$ or $\zeta$ is such that there exists $\eps>0$ with $\lambda z^2<h(z)$ for $z\in (\zeta-\eps,\zeta)$. Let $u$ be a vertex in $\tG$ whose parent $i$ belongs to the set of pivotal players in $\proj(\tG)$. Let $C_{s}(u,\boldt)$ be the final number of active vertices with degree $s\geq 0$ at the end of the symmetric threshold epidemic on $\tG$, with threshold parameter $\boldt$, when the only vertex in the initial seed is $u$. Then we have:
\begin{eqnarray*}
\frac{C_s(u,\boldt)}{\tn} &\overset{p}{\longrightarrow}&  \frac{\left[s \gamma_s+(1-\gamma_s)\right] p_s}{\tgamma} t'_{s0}(1-\zeta^{s}) + \frac{p_s}{\tgamma} \sum_{\ell\neq 0}t'_{s\ell} \left(1-\sum_{r\geq s-\ell}b_{sr}(\zeta)\right)
\end{eqnarray*}

Using the definition of $\boldt'$, we have that $t'_{s0}=t_{s0}$ and: 
\begin{eqnarray*}
 \sum_{\ell\neq 0} t'_{s\ell}\sum_{r\geq s-\ell}b_{sr}(\zeta) &=& \sum_{\ell\neq 0}  (1-\gamma_s)t_{s\ell} \sum_{r\geq s-\ell}b_{sr}(\zeta)+  \gamma_s(1-t_{s0}),
\end{eqnarray*}
which finally gives:
\begin{eqnarray*}
\frac{C_s(u,\boldt)}{\tn} \overset{p}{\longrightarrow}  \frac{\left[s \gamma_s+(1-\gamma_s)\right] p_s}{\tgamma} t_{s0}(1-\zeta^{s}) + \frac{(1-\gamma_s)p_s}{\tgamma}  \left(1-t_{s0}-\sum_{\ell\neq 0}t_{s\ell}\sum_{r\geq s-\ell}b_{sr}(\zeta)\right)
\end{eqnarray*}

Then, by an argument similar as the one at the end of Theorem \ref{th:diff} or equation \eqref{eq:pivotal_players}, we have that $u$ belongs to the set of pivotal players in $\tG$, which ends the proof.


\subsection{Proof of Theorem \ref{th:contagionA}}

We use the same idea as in the previous proof. The same statement as for Proposition \ref{prop:link_sym_thr_model} holds when the epidemic starts from a set (instead of a single vertex $u$). Indeed, let $S$ be the initial seed in $\proj(\tG)$. By definition, the initial seed $\tS$ in $\tG$ consists of the vertices whose parent belongs to $S$. 

Let $C_{s\ell}(\tS,\boldt)$ (resp. $C'_{s\ell}(S,\boldt')$) be the final number of active vertices with degree $s\geq 0$ and threshold $\ell$ at the end of the symmetric threshold epidemic on $\tG$ (resp. $\proj(\tG)$), with threshold parameter $\boldt$ (resp. $\boldt'$, defined in Proposition \ref{prop:link_sym_thr_model}), when the initial seed is $\tS$ (resp. $S$). 

Using a slight extension of Theorem 10 in \cite{lel:diff}, we have, for all $s\geq 0$ and $\ell\geq 0$:
\begin{eqnarray*}
C'_{s\ell}(S,\boldt')/n & \to_p & p_s t'_{s\ell} \left(\alpha_s+ (1-\alpha_s)\left(1-\sum_{r\geq s-\ell}b_{sr}(\zeta)\right)\right),
\end{eqnarray*}
where $\zeta$ is defined in \eqref{eq:zetaContA}. More precisely, the first term $p_s t'_{s\ell} \alpha_s$ comes from the vertices that belong to the initial seed $S$, and the second one $p_s t'_{s\ell} (1-\alpha_s)\left(1-\sum_{r\geq s-\ell}b_{sr}(\zeta)\right)$ comes from those that are activated during the process. In order to obtain the asymptotic for $C_{s\ell}(\tS,\boldt)/n$, we have to multiply the first term by $(s\gamma_s+1-\gamma_s)$. The multiplicative constant for the second term depends on the value of the threshold $\ell$: if $\ell=0$, we multiply the second term by $(s\gamma_s+1-\gamma_s)$, and if $\ell>0$, we multiply it by $1$ (since the vertices with positive threshold that are activated during the process necessarily do not belong to a clique).
Summing over $s$, $\ell$ and replacing $\boldt'$ by its expression gives the following limit, as $n\to\infty$:
\begin{eqnarray*}
C(\boldt,\balpha)/n & \to_p& \sum_{s}p_s t_{s0}  (s \gamma_s+1-\gamma_s) \Big[ \alpha_s +(1-\alpha_s)(1-\zeta^{s}) \Big] \\ 
&& +\sum_{s} p_s (1-\gamma_s) \alpha_s(s \gamma_s+1-\gamma_s)(1-t_{s0})  \\
&&+\sum_{s}  p_s(1-\gamma_s)(1-\alpha_s)\left[(1-t_{s0})-\sum_{\ell\neq 0}t_{s\ell}\sum_{r\geq s-\ell}b_{sr}(\zeta)\right] \\
&&+\sum_{s}  p_s \gamma_s (1-t_{s0})\alpha_s( \gamma_s+1-\gamma_s) .
\end{eqnarray*}
Gathering some terms and using that $\tn/n\to_p\tgamma$ ends the proof of Theorem \ref{th:contagionA}.

\section{Conclusions}


Up to out knowledge, our analysis is the first systematic study of random graphs with both a tunable asymptotic degree distribution and a clustering coefficient. Our model has the advantage to still be tractable for the analysis of diffusion or symmetric threshold model.

For both models, we are able to derive explicit formulas for the cascade condition, \ie the condition under which a single infected individual can turn a positive fraction of the population into infected individuals. When such a cascade is possible, the expression of its size is given analytically. 
In the case of random regular graphs, we proved that the clustering 'inhibits' the diffusion process. Numerical evaluations also show that clustering decreases the cascade size of the diffusion process for regular graphs, and 'inhibits' the diffusion process for power-law graphs.
The impact of clustering on the symmetric threshold model is studied in the particular case of the contagion model described in \ref{subs:contagion_clust}: numerical evaluations show that the effect of clustering on the contagion process depends on the value of the mean degree in the graph: while clustering 'inhibits' the contagion for a low mean degree, the contrary happens in the high values regime. When a cascade is possible, we observe that clustering decreases its size.

In addition, we can also compute explicitly the cascade size in the case of a degree based activation, for both diffusion and symmetric threshold models. 
This theoretical analysis paves the way to a possible control of such epidemic processes as done in \cite{bcgs10} or \cite{lel:sig09}.


\appendix

\section{Branching process approximation for the diffusion threshold (with a single activation)}
\label{app:BP}

We can guess the value of the diffusion threshold $\pi_c$ given in Theorem \ref{th:diff} using a branching process approximation. Indeed the random graph $\Gr$ can be approximated by a branching process $\Gamma$ in which each node (except the root) has a number of offspring distributed as $D^*$. We add cliques in this branching process as in $\tGr$, which gives a graph $G_{\Gamma}$. We then proceed in two steps: first we delete independently with probability $1-\pi$ (in $G_{\Gamma}$) each ``internal'' edge, \ie edge inside a clique; second we delete independently with probability $1-\pi$ (in the new graph) each ``external'' edge, \ie edge outside cliques.

\begin{center}

 \begin{tabular}{|c|c|c|}
 \hline
 & $\Gamma$ & $G_{\Gamma}$ \\
Before percolation inside cliques &\raisebox{-1.2cm}{\includegraphics[scale=0.3]{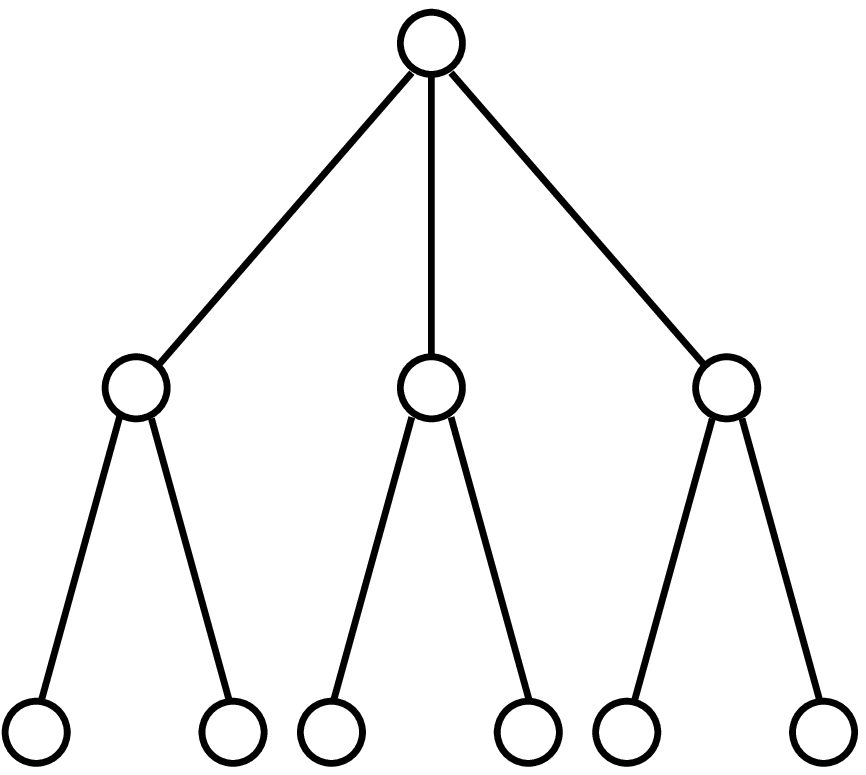}}&\raisebox{-1.2cm}{\includegraphics[scale=0.3]{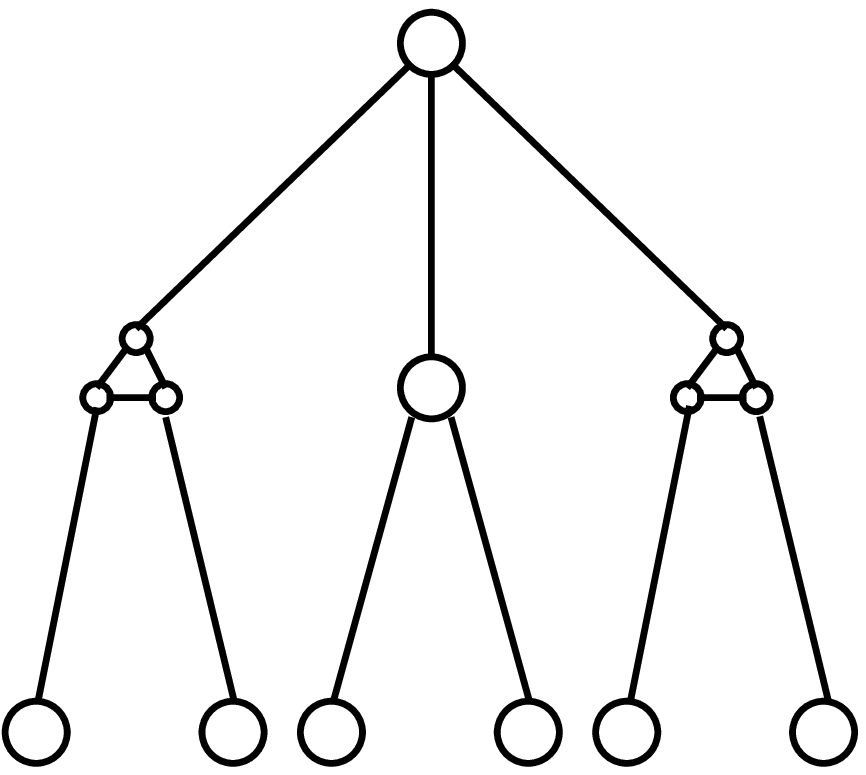}} \\
\hline
& $\Gamma'$ & $G'_{\Gamma}$ \\
After percolation inside cliques &\raisebox{-1.2cm}{\includegraphics[scale=0.3]{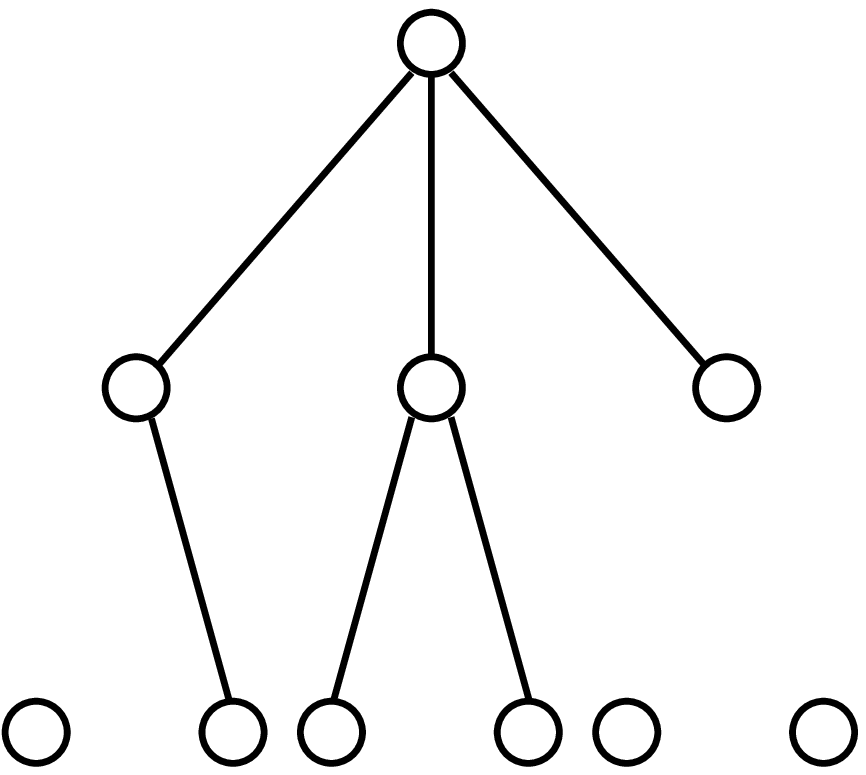}}&\raisebox{-1.2cm}{\includegraphics[scale=0.3]{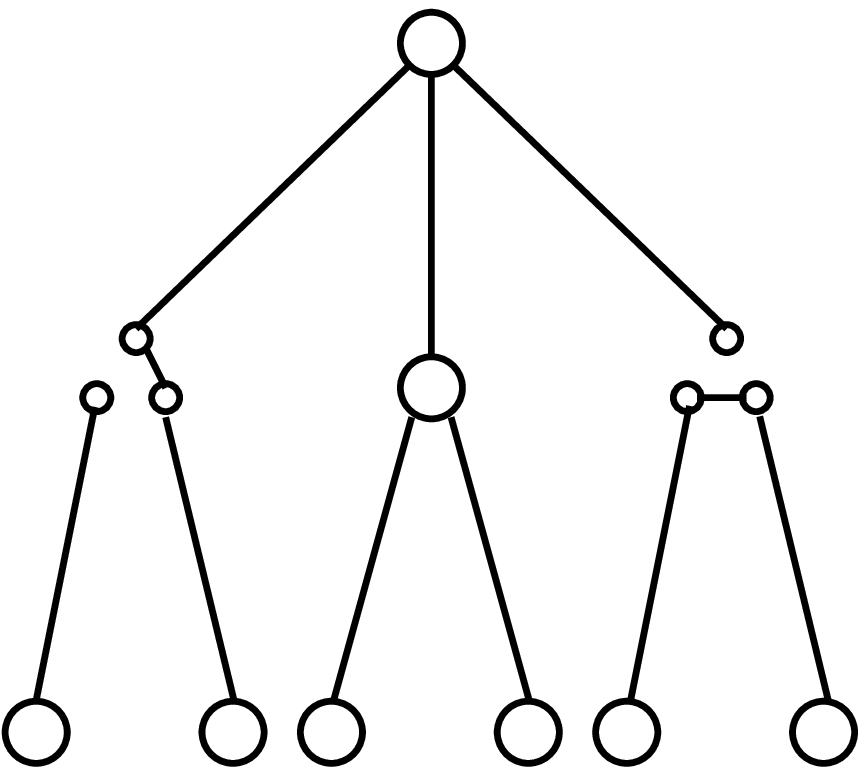}}\\
\hline
\end{tabular}

\end{center}

After the first deletion of edges, we get a new graph $G'_{\Gamma}$ in which original cliques can be broken into several components. If we make the equivalent transformation in the original branching process $\Gamma$, it means that a node can lose some of its children. 

More precisely, we consider a node $i$ in the branching process $\Gamma$. Let $e$ be the edge of $i$ that links $i$ to the previous generation. The degree of $i$ is distributed as $D^*+1$. We assume $D^*+1=d$. In $G_{\Gamma}$, node $i$ is replaced by a clique $K$ with probability $\gamma_d$. In that case, let $v$ be the vertex in $K$ whose edge outside the clique is $e$. After having deleted independently each edge inside the clique with probability $1-\pi$, the probability that the component of vertex $v$ inside $K$ contains $k$ vertices (including $v$ itself) is given by $f(d,k,\pi)$. Hence the probability that $v$ is linked to $k$ vertices (including the one linked by $e$) is: $ (1-\gamma_d) \Indb(d=k) + \gamma_d f(d,k,\pi) = \Pb\left( \cK(D^*+1,\pi,\boldg)=k\right)$. The new distribution of offspring in the corresponding branching process $\Gamma'$ is thus $\cK(D^*+1,\pi,\boldg)-1$. Finally, we remove each (external) edge with probability $1-\pi$, which gives $\pi \Eb\left[ \cK(D^*+1,\pi,\boldg)-1\right]$ for the expected number of offspring.


\begin{thebibliography}{10}

\bibitem{AOY}
D.~Acemoglu, A.~Ozdaglar, and E.~Yildiz.
\newblock Diffusion of innovations in social networks.
\newblock In {\em IEEE Conference on Decision and Control (CDC)}, 2011.

\bibitem{amini:10}
H.~Amini.
\newblock Bootstrap percolation and diffusion in random graphs with given
  vertex degrees.
\newblock {\em The Electronic Journal of Combinatorics}, 17, 2010.

\bibitem{BP:bootstrap}
J.~Balogh and B.~G. Pittel.
\newblock Bootstrap percolation on the random regular graph.
\newblock {\em Random Structures Algorithms}, 30(1-2):257--286, 2007.

\bibitem{bl95}
L.~E. Blume.
\newblock The statistical mechanics of best-response strategy revision.
\newblock {\em Games Econom. Behav.}, 11(2):111--145, 1995.
\newblock Evolutionary game theory in biology and economics.

\bibitem{bb}
B.~Bollob{\'a}s.
\newblock {\em Random graphs}, volume~73 of {\em Cambridge Studies in Advanced
  Mathematics}.
\newblock Cambridge University Press, Cambridge, second edition, 2001.

\bibitem{bcgs10}
C.~Borgs, J.~Chayes, A.~Ganesh, and A.~Saberi.
\newblock How to distribute antidote to control epidemics.
\newblock {\em Random Struct. Algorithms}, 37:204--222, September 2010.

\bibitem{britton-2007}
T.~Britton, M.~Deijfen, A.~N. Lager{\aa}s, and M.~Lindholm.
\newblock Epidemics on random graphs with tunable clustering.
\newblock {\em J. Appl. Probab.}, 45(3):743--756, 2008.

\bibitem{CL:Netgcoop}
E.~Coupechoux and M.~Lelarge.
\newblock Impact of clustering on diffusions and contagions in random networks.
\newblock In {\em NetGCooP 2011: International conference on NETwork Games,
  COntrol and OPtimization}, 2011.

\bibitem{DeijfenKets}
M.~Deijfen and W.~Kets.
\newblock Random intersection graphs with tunable degree distribution and
  clustering.
\newblock {\em Probab. Eng. Inform. Sci.}, 23:661--674, 2008.

\bibitem{gil59}
E.~N. Gilbert.
\newblock Random graphs.
\newblock {\em Ann. Math. Statist.}, 30:1141--1144, 1959.

\bibitem{GleesonMH:clustering}
J.~P. Gleeson, S.~Melnik, and A.~Hackett.
\newblock How clustering affects the bond percolation threshold in complex
  networks.
\newblock {\em Physical Review E}, 81, 2010.

\bibitem{G78}
M.~Granovetter.
\newblock {Threshold Models of Collective Behavior}.
\newblock {\em American Journal of Sociology}, 83(6):1420--1443, 1978.

\bibitem{ja09}
S.~Janson.
\newblock On percolation in random graphs with given vertex degrees.
\newblock {\em Electron. J. Probab.}, 14:no. 5, 87--118, 2009.

\bibitem{ja:simple}
S.~Janson.
\newblock The probability that a random multigraph is simple.
\newblock {\em Combinatorics, Probability and Computing}, 18(1-2):205--225,
  2009.

\bibitem{JLR:2000}
S.~Janson, T.~{\L}uczak, and A.~Rucinski.
\newblock {\em Random graphs}.
\newblock Wiley-Interscience series in Discrete Mathematics and Optimization.
  Wiley-Interscience, New York, 2000.

\bibitem{kaiser}
M.~Kaiser.
\newblock Mean clustering coefficients: the role of isolated nodes and leafs on
  clustering measures for small-world networks.
\newblock {\em New Journal of Physics}, 10, 2008.

\bibitem{kal}
O.~Kallenberg.
\newblock {\em Foundations of modern probability}.
\newblock Springer-Verlag, 2002.

\bibitem{klein}
J.~Kleinberg.
\newblock {\em {Cascading Behavior in Networks: Algorithmic and Economic
  Issues}}.
\newblock Cambridge University Press, 2007.

\bibitem{lel:sig09}
M.~Lelarge.
\newblock Efficient control of epidemics over random networks.
\newblock In J.~R. Douceur, A.~G. Greenberg, T.~Bonald, and J.~Nieh, editors,
  {\em SIGMETRICS/Performance}, pages 1--12. ACM, 2009.

\bibitem{lel:diff}
M.~Lelarge.
\newblock {Diffusion and cascading behavior in random networks.}
\newblock {\em under revision for Games and Economic Behavior},
  arxiv/1012.2062, 2010.

\bibitem{mr95}
M.~Molloy and B.~Reed.
\newblock A critical point for random graphs with a given degree sequence.
\newblock {\em Random Structures Algorithms}, 6(2-3):161--179, 1995.

\bibitem{mor}
S.~Morris.
\newblock Contagion.
\newblock {\em Rev. Econom. Stud.}, 67(1):57--78, 2000.

\bibitem{new03}
M.~E.~J. Newman.
\newblock Properties of highly clustered networks.
\newblock {\em Phys. Rev. E}, 68(2):026121, Aug 2003.

\bibitem{new-book03}
M.~E.~J. Newman.
\newblock The structure and function of complex networks.
\newblock {\em SIAM Rev.}, 45(2):167--256 (electronic), 2003.

\bibitem{new09}
M.~E.~J. Newman.
\newblock Random graphs with clustering.
\newblock {\em Phys. Rev. Lett.}, 2009.

\bibitem{trapman07}
P.~Trapman.
\newblock On analytical approaches to epidemics on networks.
\newblock {\em Theoretical Population Biology}, 71(2):160--173, 2007.

\bibitem{vega07}
F.~Vega-Redondo.
\newblock {\em Complex social networks}, volume~44 of {\em Econometric Society
  Monographs}.
\newblock Cambridge University Press, Cambridge, 2007.

\bibitem{wat02}
D.~J. Watts.
\newblock A simple model of global cascades on random networks.
\newblock {\em Proc. Natl. Acad. Sci. USA}, 99(9):5766--5771 (electronic),
  2002.

\bibitem{ws98}
D.~J. Watts and S.~H. Strogatz.
\newblock Collective dynamics of 'small-world' networks.
\newblock {\em Nature}, 393(6684):440--442, June 1998.

\end{thebibliography}
\end{document}